\documentclass{article}

\usepackage{amsmath,amssymb,amsthm}

\usepackage[T1]{fontenc}
\usepackage{libertine}
\usepackage[libertine]{newtxmath}

\numberwithin{equation}{section}
\newtheorem{thm}{Theorem}[subsection]
\newtheorem{cor}[thm]{Corollary}
\newtheorem{prop}[thm]{Proposition}
\newtheorem{lem}[thm]{Lemma}
\newtheorem{conj}[thm]{Conjecture}
\newtheorem{prob}[thm]{Problem}
\theoremstyle{definition}
\newtheorem{defn}[thm]{Definition}
\theoremstyle{remark}
\newtheorem{rmk}[thm]{Remark}
\newtheorem{exam}[thm]{Example}

\usepackage{microtype}

\usepackage[all]{xy}

\usepackage[usenames]{color}
\usepackage{todonotes}

\usepackage{hyperref}
\hypersetup{
  colorlinks   = true, 
  urlcolor     = blue, 
  linkcolor    = blue, 
  citecolor   = red 
}

\newcommand{\co}{\colon\thinspace}
\newcommand{\mb}[1]{\mathbb{#1}}
\newcommand{\mf}[1]{\mathfrak{#1}}

\newcommand{\HF}{{\rm H}{\mb F}}

\newcommand{\Unstop}{\Lambda}
\newcommand{\unstop}{\lambda}
\newcommand{\xx}{\overline{\xi}}

\newcommand{\bcQ}{{\overline{Q}}}
\newcommand{\mQ}{{\widehat Q}}

\newcommand{\too}{\xrightarrow}
\newcommand{\rH}{{\widetilde H}}

\newcommand{\pow}[1]{[\![{#1}]\!]}

\newcommand{\teth}[1]{\stackrel{#1}{\leftrightsquigarrow}}
\newcommand{\ppa}[1]{{#1}^{\pm}}
\newcommand{\wt}[1]{\widetilde{#1}}
\newcommand{\hash}{\mathbin{\#}}

\DeclareMathOperator{\Ext}{Ext}
\DeclareMathOperator{\Hom}{Hom}
\DeclareMathOperator{\Map}{Map}
\DeclareMathOperator{\Tor}{Tor}

\DeclareMathOperator{\Pic}{Pic}

\DeclareMathOperator{\BP}{BP}

\DeclareMathOperator*{\into}{\hookrightarrow}

\DeclareMathOperator*{\sma}{\wedge}
\DeclareMathOperator*{\tens}{\otimes}
\DeclareMathOperator*{\ctens}{\hat\otimes}
\DeclareMathOperator*{\Hash}{\scalebox{1.5}{\raisebox{-0.2ex}{$\#$}}}

\title{Secondary power operations and the Brown--Peterson spectrum at
  the prime $2$}
\author{Tyler Lawson\thanks{The author was partially supported by NSF
    grant 1610408.}}

\begin{document}
\maketitle

\begin{abstract}
  The dual Steenrod algebra has a canonical subalgebra isomorphic to
  the homology of the Brown--Peterson spectrum. We will construct a
  secondary operation in mod-2 homology and show that this canonical
  subalgebra is not closed under it. This allows us to conclude that
  the 2-primary Brown--Peterson spectrum does not admit the structure
  of an $E_n$-algebra for any $n \geq 12$, answering a question of May
  in the negative.
\end{abstract}

\section{Introduction}

The following appeared as Problem 1 in J.P. May's ``Problems in
infinite loop space theory'' {\cite{may-problems}}.
\begin{prob}
  \label{prob:mainprob}
  For any prime $p$, does the $p$-local Brown--Peterson spectrum $BP$
  of \cite{brown-peterson-bp} admit the structure of an
  $E_\infty$-algebra?
\end{prob}
Our goal in this paper is to address this question when $p=2$. We will
construct a secondary operation in the homology of $E_\infty$-algebras
at the prime $2$ and show, with an analysis that begins with the
calculations of Johnson--Noel \cite{noel-johnson-ptypical}, that the
homology $H_* BP$ cannot admit such a secondary operation. Thus,
Problem~\ref{prob:mainprob} has a {\em negative} answer at the prime
$2$.

\subsection{Background}

Coherently commutative multiplication structures have a long history
in homotopy theory, originating in the study of the cup product. The
cup product in the cohomology of a space $X$ comes from the structure
of a differential graded algebra on the cochains $C^*(X)$, and while
there are many variants on this algebra structure that all give rise
to the cup product there is no natural cochain-level cup product that
is graded-commutative. Instead, the cup product
$\alpha \mathop{\smile} \beta$ and its reverse
$\pm \beta \mathop{\smile} \alpha$ are chain homotopic by a natural
operation $\alpha \mathop{\smile}_1 \beta$, called the cup-1
product. The cup-1 product is not graded-commutative either, but
differs from its reverse by an operation called the cup-2 product, and
these constructions extend out both to arbitrarily high ``coherences''
(giving cup-$i$ products for all $i$) and to operations accepting
arbitrarily many inputs (giving a more complicated set of operations
of several variables discussed in \cite{mcclure-smith}). The result is
called an {\em $E_\infty$-algebra structure} and Steenrod's reduced
power operations in the cohomology of spaces are built from it
\cite{steenrod-epstein}.

Since then, these coherently commutative multiplications have been
recognized in many other areas: iterated loop spaces, monoidal
structures on categories, structures in mathematical physics related
to string theory, and multiplications in cohomology theories. By
contrast with algebra, where commutativity is simply a {\em property}
of a ring, coherent multiplications come in a hierarchy: there are
$E_1$-algebra structures that correspond to associative products and
there are $E_\infty$-algebra structures that correspond to commutative
products, but there are also $E_n$-algebra structures for
$1 < n < \infty$ that interpolate between these concepts.

When we switch from ordinary cohomology to generalized cohomology
theories, chain complexes become replaced with {\em spectra}. A {\em
  ring spectrum} $R$ is a representing object for a cohomology theory
$R^*(-)$ so that cohomology with coefficients in $R$ naturally takes
values in rings. This was refined to the concept of an $E_n$-algebra
structure on the spectrum $R$ in \cite[I.4]{bmms-hinfty}, and these
more refined algebras have come to occupy a central role because
$E_n$-algebra structures produce concrete tools that are not available
to an ordinary ring spectrum
\cite{mandell-derivedsmash,lurie-higheralgebra}.
\begin{itemize}
\item An $E_1$-algebra $R$ can be given categories of {\em left
    $R$-modules} and {\em right $R$-modules}, whose homotopy
  categories are triangulated categories. These enjoy several forms of
  compatibility as $R$ varies, extend to categories of bimodules, and
  have relative smash products $\sma_R$ with properties much like the
  tensor product.
\item The category of left modules over an $E_2$-algebra $R$ is
  canonically equivalent to the category of right modules, and the
  smash product $\sma_R$ makes the category of left $R$-modules into a
  monoidal category. The homotopy category of left $R$-modules has the
  structure of a (neither symmetric nor braided) tensor triangulated
  category.
\item The homotopy category of left modules over an $E_3$-algebra has
  the structure of a braided monoidal category.
\item The homotopy category of left modules over an $E_4$-algebra has
  the structure of a symmetric monoidal category.
\item The category of modules over an $E_\infty$-algebra $R$ has
  homotopy-theoretic versions of symmetric power operations, making
  it possible to discuss a relative version of the above: we can
  define $E_n$ $R$-algebras which satisfy all of the above
  properties.
\item An $E_\infty$-algebra $R$ has, for any principal $\Sigma_n$-bundle
  $P \to B$, natural {\em geometric power operations} $R^0(X) \to R^0(P
  \times_{\Sigma_n} X^n)$ and $R^0(X) \to R^0(B \times X)$ in
  $R$-cohomology that enhance the multiplicative structure. When $P
  = \Sigma_n$ these recover the operation that sends a class to its
  $n$th power.
\end{itemize}
Many examples of $E_\infty$-algebras exist. Commutative rings $A$
produce $E_\infty$-algebras $HA$ via the Eilenberg--Mac Lane
construction; the spectra $KO$ and $KU$, representing real and complex
$K$-theory, have $E_\infty$-algebra structures whose origin is the
tensor product of vector bundles; bordism spectra like $MO$, $MSO$,
$MU$, and the like have $E_\infty$-algebra structures whose origin is
the product structure on manifolds; if $Y$ is an infinite loop space,
then there is a spherical group algebra $\mb S[Y] = \Sigma^\infty_+ Y$
with an $E_\infty$-algebra structure; and if $R$ is an
$E_\infty$-algebra and $X$ is a space, there is a spectrum $R^X$
(playing the role of ``cochains on $X$ with coefficients in $R$'')
with an $E_\infty$-algebra structure that combines the multiplication
on $R$ with the diagonal map $X \to X \times X$.

Problem~\ref{prob:mainprob} dates back to the first systematic studies
of $E_\infty$-algebras. Understanding why this result is so desirable
requires knowing a little about what the Brown--Peterson spectrum is
and how important it is in stable homotopy theory.

The complex bordism spectrum $MU$ has an $E_\infty$-algebra structure
and it is central to Quillen's relation between stable homotopy theory
and formal group laws \cite{quillen-fgl}, which initiated the subject
of chromatic homotopy theory. However, while almost the entirety of
chromatic theory is possible to phrase in terms of $MU$, the
$p$-localization $MU_{(p)}$ decomposes into summands equivalent to
this irreducible Brown--Peterson spectrum $BP$. The Brown--Peterson
spectrum has simpler cohomology and homotopy groups than $MU$ and has
canonical descriptions that are internal to the stable homotopy
category \cite{priddy-cellularBP}. The Brown--Peterson spectrum also
exhibits the close connection between $p$-local stable homotopy theory
and the theory of formal group laws, but with the added benefit that
nearly every deep structural property of chromatic homotopy theory or
formal group law theory is made more concise and more conceptually
accessible through the eyes of $BP$-theory (see, for example,
\cite{ravenel-greenbook} for extensive applications).

The existence of an $E_\infty$-algebra structure on $BP$ would be
useful in several ways.
\begin{itemize}
\item The Adams--Novikov spectral sequence is a method for calculating
  the set of homotopy classes of maps between two spectra $X$ and $Y$
  and can be derived from either their $MU$-homology or
  $BP$-homology. The computational tools using $MU$-theory (such as
  the cobar complex) are well behaved with respect to the geometric
  power operations discussed earlier, which appear in places such as
  the construction of manifolds of Kervaire invariant one
  \cite{bruner-extendedpowers}. If $BP$ had an $E_\infty$-algebra
  structure then computations of these geometric power operations
  using $MU$-theory could instead be related to simpler computations
  in $BP$-theory.
\item Such a structure would allow more concise constructions of many
  important objects in chromatic theory, such as the Morava
  $K$-theories $K(n)$ and the truncated Brown--Peterson spectra
  $\BP\langle n\rangle$, as $BP$-algebras rather than as
  $MU$-algebras.
\item These algebra structures would mean that several computations
  with these ring spectra could be governed by the computations for
  $BP$-theory, such as computations of topological Hochschild homology
  and topological cyclic homology that are important to current work
  in algebraic $K$-theory \cite{ausoni-rognes}. These can also be
  extended by relative computations in $BP$-modules, which are much
  simpler than the relative calculations in $MU$-modules.
\item Perhaps most importantly, the Brown--Peterson spectrum is one of
  the most prominent examples of an important homology theory where
  our knowledge of geometric interpretations (e.g. via Baas--Sullivan
  theory \cite{baas-singularitybordism}) lags far behind our algebraic
  knowledge.\footnote{From \cite{may-problems}: ``The point here is
    that the notion of an $E_\infty$ ring spectrum seems not to be a
    purely homotopical one; good concrete geometric models are
    required, and no such model is known for $BP$.''} Many of the
  prominent examples of $E_\infty$-algebras, such as $K$-theory and
  bordism theory, originate in cohomology theories with geometric
  cycles or cocycles that have a product. The existence of an
  $E_\infty$-algebra structure on $BP$ would be a good indicator that
  a strong geometric interpretation existed.
\end{itemize}

This problem has generated a great deal of interesting research. The
existence of multiplication structures in the homotopy category has a
long history (for example, see the introduction of
\cite{strickland-products}). Several forms of obstruction theory have
been developed which showed that many spectra constructed by
Baas--Sullivan theory admit $E_1$-algebra structures
\cite{robinson-moravaassociativity, lazarev-ainfty,
  baker-jeanneret-hopfalgebroid, angeltveit}. More sophisticated
obstruction theory has appeared for $E_\infty$-algebras
\cite{robinson-gammahomology,goerss-hopkins-summary}, and Richter
obtained lower bounds on the amount of commutativity present in $BP$
based on Robinson's obstruction theory
\cite{richter-BPcoherences}. Techniques such as localization and
idempotent splitting were developed in \cite{may-idempotent} to handle
additive and multiplicative versions of the construction of $BP$. More
recently Basterra--Mandell showed that $BP$ is a split summand of
$MU_{(p)}$ as an $E_4$-algebra \cite{basterra-mandell-BP}, and so the
homotopy category of $BP$-modules has a symmetric monoidal structure;
Chadwick--Mandell used idempotent splittings to show that this could
be done with the Quillen idempotent as $E_2$-algebras
\cite{chadwick-mandell-genera}. Both Hu--Kriz--May
\cite{hu-kriz-may-cores} and Baker \cite{baker-closeencounters} gave
iterative constructions by methods that kill torsion, producing two
different types of closest possible torsion-free $E_\infty$-algebra to
$BP$. An unpublished paper of Kriz attempted to prove that $BP$ admits
an $E_\infty$-algebra structure, and Basterra developed the theory of
topological Andr\'e-Quillen (TAQ) cohomology based on his ideas---this
theory allows the construction of $E_\infty$-algebras by
systematically lifting the $k$-invariants in the Postnikov tower from
ordinary cohomology to TAQ-cohomology
\cite{kriz-towers,basterra-andrequillen}. Kriz's original program
foundered on a technical detail, but TAQ has been central in a great
deal of research since.\footnote{The problem, insofar as the author
  understands, was establish certain elements in the Miller spectral
  sequence computing TAQ-cohomology needed to be shown to be permanent
  cycles, but the operations used to establish this were
  insufficiently compatible with the differentials.}

However, the hope that Problem~\ref{prob:mainprob} has a positive
solution perhaps originated in a time of much greater optimism, and
the intervening years have shown that the additive and multiplicative
structure of a spectrum are difficult to untangle from each
other. Indeed, there is something closer to a reciprocity
relationship, where requirements of the additive structure are
rewarded with constraints on the multiplicative and vice versa. In
line with this, there have been several more recent calculations
showing that desirable properties of a multiplication on $BP$ cannot
be realized. Hu--Kriz--May showed that there cannot be a map of
$E_\infty$-algebras $BP \to MU_{(p)}$ because it conflicts with
calculation of Dyer--Lashof operations in their homology, despite the
presence of the Quillen idempotent which describes such a splitting
additively and algebraically \cite{hu-kriz-may-cores}.  In the reverse
direction, Johnson--Noel showed with hard calculations that the
particular map of ring spectra $MU_{(p)} \to BP$ employed to great
effect in chromatic theory cannot be a map of $E_\infty$-algebras for
$p \leq 13$ \cite{noel-johnson-ptypical}, based on a power operation
criterion due to McClure \cite[VIII.7.7, 7.8]{bmms-hinfty}.

The Hu--Kriz--May result seems more decisive mainly because it uses
structure that is forced. The mod-$p$ homology groups $H_* BP$ are
identified as a canonical subalgebra of the dual Steenrod algebra
$H_* \HF_p$, and this means that the ring structure on $H_* BP$ and
operations coming from any $E_n$-algebra structure (including the
Dyer--Lashof operations mentioned above) are completely determined by
those in the dual Steenrod algebra. It is straightforward to show that
$H_* BP$ is closed in $H_*\HF_p$ under the Dyer--Lashof operations,
and so we cannot exclude the possibility that $BP$ is an
$E_\infty$-algebra using a relation between these primary
operations. This paper shows that, at the prime $2$, there does exist
a contradiction for a more subtle reason: while $H_* BP$ is closed
under primary operations, it is not closed under secondary
operations. This parallels Adams' solution of the Hopf invariant one
problem using secondary cohomology operations \cite{adams-J(X)IV}. The
proof will critically rely on Johnson--Noel's calculation of power
operations in complex bordism.

\begin{thm}[{\ref{prop:bigoperation}, \ref{cor:A-detectsagain}}]
  \label{thm:mainsecondary}
  There exists a natural secondary operation in the mod-$2$ homology
  of $E_{12}$-algebras with the following properties. For $R$ an
  $E_{12}$-algebra, the secondary operation is defined on the subset of
  $H_2 R$ satisfying certain identities between Dyer--Lashof
  operations in $H_k R$ for $5 \leq k \leq 13$, and the secondary
  operation takes values in a quotient of $H_{31} R$. This operation
  is preserved by maps $R \to R'$ of $E_{12}$-algebras.

  In the dual Steenrod algebra $H_*\HF_2 \cong \mb
  F_2[\xi_1,\xi_2,\dots]$, this operation is defined on the element
  $\xi_1^2 \in H_2(\HF_2)$ and, mod decomposables, unambiguously takes
  the value $\xi_5 \in H_{31}(\HF_2)$.
\end{thm}

With this theorem, we can exclude the existence of an
$E_\infty$-algebra structure on the $2$-local Brown--Peterson spectrum
and several related objects (e.g. the generalized $BP\langle k\rangle$
whose cohomology is discussed in \cite[4.3]{tmforientation} and whose
additive uniqueness is discussed in \cite{angeltveit-lind-bpn}).
\begin{thm}[{\ref{thm:mainthm-redux}, \ref{thm:nobp-redux}}]
  \label{thm:nobp}
  Suppose that $R$ is a connective $E_{12}$-algebra with a ring
  homomorphism $\pi_0 R \to \mb F_2$ such that the induced map on
  mod-$2$ homology $H_* R \to H_* \HF_2$ is injective in degrees $5$
  through $13$. If $\xi_1^2$ is in the image of $H_2(R)$, then the
  element $\xi_5$ is in the image of $H_{31} R$ mod decomposables.

  In particular, the $2$-local Brown--Peterson spectrum $BP$, the
  (generalized) truncated Brown--Peterson spectra $BP\langle k\rangle$
  for $k \geq 4$, and their $2$-adic completions do not admit the
  structure of $E_n$-algebras for any $12 \leq n \leq \infty$.
\end{thm}

\subsection{Remarks on obstruction theory}

The secondary operation we will define is determined by a relation
between Dyer--Lashof operations (the full relation is rather large,
and is displayed in Proposition~\ref{prop:bigrelation}). For us, this
relation is not obvious; it is not obvious that this particular
relation is relevant; and it is not obvious that the resulting
secondary operation is calculable. We did not find this relation by
trial and error---or, more accurately, we tried to find relevant
secondary operations by trial and error and failed. All of our
preliminary attempts resulted in combinations that were excluded by
necessity of compatibility with the Steenrod operations. In this
section we will indicate a little about how the main result of this
paper was found, as opposed to how it is written.\footnote{A more detailed
explanation of this calculation is now in \cite{bpobstructions}.}

The obstruction theory of Goerss--Hopkins
\cite{goerss-hopkins-moduliproblems} takes as input a simplicial
operad, an appropriate homology theory $E_*$, and an algebra $A$ for
this simplicial operad in $E_* E$-comodules. From this, it produces an
obstruction theory to calculate the moduli space of algebras over the
geometric realization of this operad whose $E$-homology is $A$. Senger
specialized this to the case where $E$ is mod-$p$ homology and the
operad is a constant $E_\infty$-operad \cite{senger-obstr}. His work
produced an obstruction theory whose input is an algebra $A$ with
Steenrod operations and Dyer--Lashof operations satisfying instability
relations and Nishida relations, whose obstruction groups are
$\Ext$-groups in this category, and which calculated the moduli space
of $E_\infty$-algebras whose mod-$p$ homology is $A$. He also
developed several tools for reducing these calculations to more
tractable $\Ext$-groups that could, in the case of $BP$ or
$BP\langle n\rangle$, be calculated with a Koszul resolution
\cite{priddy-koszul}. By construction, this obstruction theory
remembers that the Nishida relations will exclude a number of possible
obstructions. (The problem encountered in Kriz's preprint
\cite{kriz-towers} could be viewed as the accidental exclusion of too
many obstructions in this fashion.)

In the case of the $2$-primary Brown--Peterson spectrum, calculations
with this obstruction theory indicated two first potential nonzero
obstruction classes. We can define $y$, $R^n$, and $v_m$ to be,
respectively, $\Ext$-Koszul dual to the generator
$\xi_1^2 \in H_2(BP)$, the Dyer--Lashof operation $Q^{n-1}$, and the
Milnor primitive $Q_{m-1}$ in the Steenrod algebra. (The $R^n$ are
closely related to unpublished work of Basterra--Mandell on operations
in TAQ-cohomology.) Then, using this notation, the first possible
obstruction classes are $v_3^2 R^{19} R^9 y$ and $v_4 R^{21} R^9
y$. Under the yoga of secondary operations described by Adams
\cite{adams-J(X)IV}, the potential obstruction class
$v_4 R^{21} R^9 y$ would detect an obstruction from a secondary
operation whose value involved $\xi_5$ (detected by the Milnor
primitive), combining relations that (at least) involved the Adem
relations for $Q^{20} Q^8$ and an identity satisfied by $Q^8 \xi_1^2$.

Indeed, our main result is that this is the case. However, much of the
progress in this paper traces its origin back to the actual
calculation of these relations. After determining the needed
identities in Proposition~\ref{prop:bigrelation}, we could identify
most of the relations in $H_* BP \subset H_* \HF_2$ as already holding
true in $H_* MU$, making it possible to begin juggling this secondary
operation through much simpler ones passing through $H_* MU$.

\subsection{Further questions}

In the original version of this paper, we expressed the following
strong belief that the $2$-primary Brown--Peterson spectrum was not
unique in failing to admit an $E_\infty$-algebra structure.
\begin{conj}
  \label{conj:counterconj}
  For any odd prime $p$, the $p$-local Brown--Peterson spectrum does
  not admit the structure of an $E_\infty$-algebra.
\end{conj}
Senger has already extended the methods of this paper to prove
Conjecture~\ref{conj:counterconj}, showing that $BP$ (and
$BP\langle n\rangle$ for $n \geq 4$) do not admit the structure of
$E_{2(p^2 + 2)}$-algebras at any prime $p$ \cite{senger-bp}.

Our keystone computation in this paper is a Dyer--Lashof operation in
a version of the 2-primary dual Steenrod algebra for $MU$-modules.
\begin{prob}
  Determine how the Dyer--Lashof operations act on the $p$-primary
  $MU$-dual Steenrod algebras $\pi_*(\HF_p \sma_{MU} \HF_p)$.
\end{prob}
Baker has shown in \cite{baker-power-operations-coactions} how to
derive the Nishida relations, describing the interaction between
cohomology operations and Dyer--Lashof operations, from the
Dyer--Lashof operations in the ordinary dual Steenrod algebra. This
suggests that a solution to the previous problem would give additional
constraints on $MU$-algebras by describing additional relations that
have to hold in their mod-$p$ homology relative to $MU$.
\begin{prob}
  Determine analogues of Nishida relations between the homology
  operations on $H_*^{MU}R = \pi_* (\HF_p \sma_{MU} R)$ and the
  Dyer--Lashof operations.
\end{prob}
In particular, Remark~\ref{rmk:noorientation} describes how the
Dyer--Lashof operation that we have calculated seem to place a cap on
multiplicative structure for the map $MU \to BP$ at the prime $2$---a
stronger cap than the one we have shown for the amount of
multiplicative structure on $BP$.
\begin{prob}
  Find constraints on the values of $n$ for which the $p$-local
  Brown--Peterson spectrum can admit the structure of an $E_n$
  $MU$-algebra.
\end{prob}
Again, in the time since we raised this question, Senger has shown
that $BP$ does not admit the structure of an $E_{2p+3}$
$MU$-algebra at any prime $p$ \cite{senger-bp}.

The calculations of this paper deduce our unexpected Dyer--Lashof
operation in the $MU$-dual Steenrod algebra from a multiplicative
Dyer--Lashof operation in the Hopf ring for $MU$, and an induced
operation in the homology of the space $SL_1(MU) \subset GL_1(MU)$ of
strict units. This is a first step towards determining the homology of
the spectrum $gl_1(MU)$, about which very little is known, using the
Miller spectral sequence \cite{miller-delooping}.
\begin{prob}
  Determine multiplicative Dyer--Lashof operations in the Hopf
  ring for $MU$ and in the homology of $GL_1(MU)$. Determine 
  homology groups of the unit spectrum $gl_1(MU)$ and the Picard
  spectrum $pic(MU)$, as well as information about their homotopy
  types.
\end{prob}

Remark~\ref{rmk:infinitybrackets} points out that our description of
secondary operations and Toda brackets is not optimal. For example, it
sometimes requires strict basepoints for mapping spaces, strict
unitality, and strict initial and terminal objects, all of which are
not invariant under unbased homotopy equivalences between objects and
not invariant under Dwyer--Kan equivalences between topological
categories. However, the tools should apply in much wider generality;
investigations in this direction have been carried out by Bhattacharya
and Hank.
\begin{prob}
  Develop a homotopical framework for secondary operations.
\end{prob}
For example, a combinatorial framework analogous to quasicategories
that encodes the notion of a category enriched in based spaces,
equivalent to that introduced by Gepner--Haugseng
\cite{gepner-haugseng-enriched}, would be extremely useful in this
direction. Ideally, this should make Cohen--Jones--Segal's
construction of filtered spectra from coherent chain complex objects
\cite[\S 5]{cohenjonessegal-floerhomotopy} part of an equivalence
between filtered objects and coherent chain complex objects in a
stable $\infty$-category, extending Lurie's version of the Dold--Kan
correspondence \cite[1.2.4]{lurie-higheralgebra}.

Our calculations with power operations in the Hopf ring make use of
the $H_\infty^2$-algebra structure on $MU$, a concept from
\cite{bmms-hinfty} that has been largely neglected in the modern
literature. It should be possible to describe a fully coherent version
of this structure using the language of Picard spaces and Picard
spectra \cite{mathew-stojanoska-pictmf}.
\begin{prob}
  Give a systematic development of $E_\infty^d$-algebras as homotopy
  coherent versions of $H_\infty^d$-ring spectra, and show that the
  $H_\infty^d$-structures on classical Thom spectra constructed in
  \cite[VIII.5.1]{bmms-hinfty} lift to $E_\infty^d$-algebra structures.
\end{prob}
An $E_\infty^d$-structure on an $E_\infty$-algebra $R$ should be a
lift of the map of spaces
\[
  d\mb Z \subset \mb Z \subset \Pic(\mb S) \to \Pic(R)
\]
to a map of $E_\infty$-spaces, corresponding to a functor of symmetric
monoidal $\infty$-categories. In close analogy with the work of
Ando--Blumberg--Gepner--Hopkins--Rezk \cite{abghr-infinity-bundles},
the point $\{dk\} \into \Pic(\mb S)$ representing $S^{dk}$ gives rise
(via the $E_\infty$-space structure) to a diagram
$B\Sigma_m \to \Pic(\mb S)$, whose homotopy colimit is a Thom
spectrum on $B\Sigma_m$ representing the extended power construction
on $S^{dk}$. An $E_\infty^d$-structure on $R$ should then make the
resulting diagram $B\Sigma_m \to \Pic(R)$ factor through a constant
diagram with value $S^{dkm}$, allowing us to conclude that the smash
product of $R$ with the Thom spectrum has an equivalence to 
$R \sma (B\Sigma_m)_+ \sma S^{dkm}$.

\subsection{Outline of proof}

We will begin by calculating a Dyer--Lashof operation in the $MU$-dual
Steenrod algebra $\pi_* (H \sma_{MU} H)$, where $H$ is the
Eilenberg--Mac Lane spectrum $\HF_2$. This has maps in from the dual
Steenrod algebra $\pi_* (H \sma H)$ and out to the homology of $SU$
which become a left exact sequence
\[
0 \to Q\pi_* (H \sma H) \to Q \pi_*\left(H \sma_{MU} H\right) \to Q H_* SU
\]
on indecomposables. The Dyer--Lashof operations on the left are known
by work of Steinberger, and were employed by Tilson
\cite{tilson-kunneth} to calculate operations in the middle term. The
Dyer--Lashof operations on the right are known by work of Kochman. The
operation we will calculate is the first possible hidden extension
and it turns out to be nontrivial.

To carry out this calculation we rely on calculations of {\em
  unstable} multiplicative Dyer--Lashof operations. This uses
Ravenel--Wilson's description of Hopf ring structure on the homology
of the spaces in the $\Omega$-spectrum for $MU$
\cite{ravenel-wilson-hopfring} and a comparison between Dyer--Lashof
operations and the tom Dieck--Quillen power operations in
$MU$-cohomology. The relevant portion of this extension is ultimately
determined by the calculation of Johnson--Noel discussed earlier
\cite{noel-johnson-ptypical}. We will make extensive use of the
results of Bruner--May--McClure--Steinberger in doing this calculation
\cite{bmms-hinfty}.

We will then give an alternative description of this operation in the
$MU$-dual Steenrod algebra as a Dyer--Lashof operation applied to the
result of a secondary operation. This allows us to use juggling
formulas for secondary operations to determine a more complicated
secondary operation in the dual Steenrod algebra, showing that
$H_* BP$ is not closed under secondary operations. (We are fortunate
in this regard that most of our calculations can be carried out mod
decomposable elements.) In order to work with this we will describe a
framework for secondary operations in
Section~\ref{sec:secondary-operations} based on Harper's book
\cite{harper-secondaryoperations}, with our emphasis shifted from
suspension and loop operators to loops inside mapping spaces.

\subsection{Terminology}

The notation $\Map$ always denotes a space, or simplicial set, of
maps. We will refer to a diagram as {\em homotopy commutative} if it
commutes in the homotopy category, and {\em homotopy coherent} if we
have further chosen compatible homotopies and higher homotopies to
recover a coherent diagram \cite{vogt-hocolim,lurie-htt}.

We will adhere to the standard conventions for function composition
and path composition, even though they make no sense. Maps in a
category are written using arrows $X \to Y$, and given $f\co X \to Y$
and $g\co Y \to Z$ there is a composite $gf$. Paths in a space are
written using arrows $p \Rightarrow q$, and given
$h\co p \Rightarrow q$ and $k\co q \Rightarrow r$ there is a path
composite $h \cdot k$.

Throughout this paper, we will write $H_*(X;R)$ for the homology
groups of $X$ with coefficients in $R$, and similarly for
cohomology. If $R$ is not specified, we view these as being taken with
coefficients in a fixed finite field $\mb F_p$ of prime
order. Homology and cohomology groups of spaces are unreduced unless
otherwise specified.

When $X$ is a spectrum, $\pi_n(X)$ always denotes the set of maps $S^n
\to X$ in the stable homotopy category.

We will let $MU$ be the complex cobordism spectrum and $F$ be the formal
group law of $MU$, writing it as
$F(x,y) = x +_F y = \sum a_{i,j} x^i y^j$ with
$a_{i,j} \in \pi_{2(i+j-1)} MU$.

\subsection{Framework}

We are in the position that we require tools from both classical and
modern frameworks.

In Section~\ref{sec:secondary-operations}, we will require highly
structured categories of algebras, well-behaved adjunctions between
them, relative smash products, and the like. To our knowledge, the
only literature that accommodate our needs for $E_n$-algebras is due
to Elmendorff--Mandell \cite{elmendorf-mandell-loopspace}, which works
in the category of symmetric spectra of with the positive stable model
structure \cite{hovey-shipley-smith-symmetric, mmss}. We will use the
term {\em commutative ring spectrum} for a commutative monoid in
symmetric spectra, and the term {\em $E_n$-algebra} for an algebra
over a fixed $E_n$-operad in simplicial sets---for this it is
convenient to use the $E_\infty$-operad of Barratt--Eccles
\cite{barratt-eccles-operad} with its filtration by $E_n$-suboperads
due to Berger \cite{berger-barratteccles}. In this framework,
Elmendorf--Mandell show that each category of $E_n$-algebras is a
simplicial model category. For $m \leq n$, the forgetful functors from
$E_n$-algebras to $E_m$-algebras or to symmetric spectra are right
Quillen functors, and there is a Quillen equivalence between
$E_\infty$-algebras and commutative ring spectra \cite[1.3,
1.4]{elmendorf-mandell-loopspace}.

In Section~\ref{sec:hopfrings} and beyond, where we are calculating
with $MU$ and the dual Steenrod algebra, we require classical results:
particularly results of Cohen--Lada--May
\cite{cohen-lada-may-homology}, May--Quinn--Ray
\cite{may-quinn-ray-ringspectra}, Bruner--May--McClure--Steinberger
\cite{bmms-hinfty}, and Ravenel--Wilson
\cite{ravenel-wilson-hopfring}. All of these results rest on the
interaction between a (possibly highly structured) ring spectrum $E$
and the spaces $E_n$ in an $\Omega$-spectrum representing it, an item
not immediately available in the positive stable model structure. Most
of these references use more classical categories of spectra, such as
those from \cite{lewis-may-steinberger}. In particular, comparisons
are easiest to draw to the $\mb S$-modules of \cite{ekmm}, and these
all have homotopically equivalent notions of commutative ring spectra
as shown in \cite{mmss}. This gives us a path to show that operations
and relations between them that we construct in
Section~\ref{sec:secondary-operations} can be related to our
calculations. (We do not mean to assert that the constructions in
Section~\ref{sec:secondary-operations} cannot be carried out within
$\mb S$-modules. To our knowledge, ours is the shortest path without
the hard work involved in creating an equivalent of
\cite{elmendorf-mandell-loopspace}.)

\subsection{Acknowledgements}

The author has benefited from discussions and perspectives provided by
many people in the course of developing this paper. The author would
particularly like to thank
Andrew Baker,
Clark Barwick,
David Benson,
Andrew Blumberg,
Robert Bruner,
John R. Harper,
Fabien Hebestreit,
Mike Hill,
Paul Goerss,
Weinan Lin,
Michael Mandell,
Peter May,
Haynes Miller,
Ulrich Pennig,
Charles Rezk,
Andrew Senger,
Neil Strickland,
Markus Szymik,
Sean Tilson,
Craig Westerland,
Dylan Wilson,
and
Steven Wilson
for their assistance. The anonymous referee also provided a great deal
of useful feedback on this paper.

\section{Secondary operations}
\label{sec:secondary-operations}

A secondary composite is the first basic type of obstruction
encountered when lifting a homotopy commutative diagram to a homotopy
coherent diagram.
\begin{defn}
  \label{def:secondarycomp}
  Let $\mathcal{D}$ be a category enriched in spaces. Suppose that we
  are given the following data:
  \begin{enumerate}
  \item a sequence  $(X_0,X_1,X_2,X_3)$ of objects of $\mathcal{D}$,
  \item maps $f_{ij}\co X_i \to X_j$ for $i < j$, and
  \item paths $h_{ijk}\co f_{jk} f_{ij} \Rightarrow f_{ik}$ in
    $\Map_\mathcal{D}(X_i, X_k)$ for $i < j
    < k$.
  \end{enumerate}
  Then the associated {\em secondary composite} is the element of
  $\pi_1 (\Map_\mathcal{D}(X_0,X_3), f_{03})$ represented by the path composite
  \[
  (h_{023})^{-1} \cdot (f_{23} h_{012})^{-1} \cdot (h_{123} f_{01})
  \cdot h_{013},
  \]
  viewed as a loop based at $f_{03}$.
  \[
    \xymatrix{
      f_{03} \ar@{=>}^-{h_{023}^{-1}}[r] &
      f_{23} f_{02} \ar@{=>}[d]^-{f_{23} h_{012}^{-1}} \\
      f_{13} f_{01} \ar@{=>}[u]^-{h_{013}} &
      f_{23} f_{12} f_{01} \ar@{=>}[l]^-{h_{123} f_{01}}
    }
  \]
\end{defn}
(In Remark~\ref{rmk:infinitybrackets} we will discuss a quasicategorical
expression of this data.)

In the following sections we will describe secondary composites which
are comparable with Massey products or Toda brackets; they rely on the
existence of distinguished ``null'' maps so that we can make sense of
composites being trivial. Our perspective is based on Harper's book
\cite{harper-secondaryoperations}.

\subsection{Secondary operations and brackets}

Throughout this section, let $\mathcal{C}$ be a category enriched in
pointed spaces (or, with appropriate modifications, pointed simplicial
sets) under $\sma$, and write $\Map_\mathcal{C}(x,y)$ for the mapping
space between any pair of objects of $\mathcal{C}$. We refer to the
basepoint of this mapping space as the {\em null map} or $\ast$; null
maps satisfy $f \ast = \ast f = \ast$ for any $f$.\footnote{Strictly
  speaking, the smash product on pointed spaces is nonassociative and
  so does not give rise to a monoidal category \cite[\S
  1.7]{may-sigurdsson-parametrized}. We really mean that we are
  working in an appropriate ``convenient category,'' such as compactly
  generated spaces.}

\begin{defn}
  \label{def:tethering}
  Suppose we have maps
  \[
  X_0 \too{f} X_1 \too{g} X_2
  \]
  in $\mathcal{C}$. A {\em tethering} of this composite is a homotopy
  class of {\em nullhomotopy} of $gf$: a homotopy class of path
  $h\co gf \Rightarrow \ast$ in $\Map_\mathcal{C}(X_0,X_2)$
  (cf. \cite[4.1.2]{harper-secondaryoperations}). We will write
  $g \teth{h} f$ to indicate such a tethering, and $g \teth{} f$ to
  indicate that there is a chosen tethering which is either implicit
  or not important to name.
\end{defn}

\begin{rmk}
  \label{rmk:tripletether}
  If a triple composite $kgf$ is nullhomotopic, then a tethering $kg
  \teth{h} f$ is the same data as a tethering $k \teth{h} gf$.
\end{rmk}

\begin{defn}
  \label{def:tetheredbracket}
  Suppose we have maps
  \[
  X_0 \too{f_{01}} X_1 \too{f_{12}} X_2 \too{f_{23}} X_3,
  \]
  and tetherings $f_{23} \teth{h_{123}} f_{12} \teth{h_{012}}
  f_{01}$. Then we define the element
  \[
  \langle f_{23} \teth{h_{123}} f_{12} \teth{h_{012}} f_{01}\rangle
  \in \pi_1 (\Map_\mathcal{C}(X_0,X_3),\ast)
  \]
  to be the path composite $(f_{23} h_{012})^{-1} \cdot h_{123}
  f_{01}$ obtained by gluing together the two nullhomotopies $f_{23}
  f_{12} f_{01} \Rightarrow \ast$. This is the secondary composite, as
  in Definition~\ref{def:secondarycomp}, obtained by choosing $f_{02}
  = f_{03} = f_{13} = \ast$ and the trivial nullhomotopies $h_{013}$
  and $h_{023}$.
\end{defn}

\begin{defn}
  \label{def:secondaryop}
  Suppose we have maps
  \[
  X_0 \too{f_{01}} X_1 \too{f_{12}} X_2 \too{f_{23}} X_3.
  \]
  If we have chosen a tethering $f_{23} \teth{h} f_{12}$ and $f_{12}
  f_{01}$ is nullhomotopic, we write
  \[
  \langle f_{23} \teth{h} f_{12}, f_{01}\rangle \subset \pi_1 (\Map_\mathcal{C}(X_0,X_3),\ast)
  \]
  for the set of all elements $\langle f_{23} \teth{h} f_{12} \teth{k}
  f_{01}\rangle$ as $k$ ranges over possible tetherings, and refer to
  $\langle f_{23} \teth{h} f_{12}, -\rangle$ as the {\em secondary
    operation} determined by the tethering. The set of maps $f_{01}$
  such that $f_{12} f_{01}$ is nullhomotopic is referred to as the
  {\em domain of definition} of this secondary operation, and the
  possibly multivalued nature of this function as the {\em
    indeterminacy} of the secondary operation.

  The secondary operations $\langle - , f_{12} \teth{}
  f_{01}\rangle$ are defined in the same way.
\end{defn}

\begin{defn}
  \label{def:brackets}
  Suppose we have maps
  \[
  X_0 \too{f_{01}} X_1 \too{f_{12}} X_2 \too{f_{23}} X_3
  \]
  such that the double composites $f_{23} f_{12}$ and $f_{12} f_{01}$
  are nullhomotopic. We define the subset
  \[
  \langle f_{23}, f_{12}, f_{01}\rangle \subset \pi_1 (\Map_\mathcal{C}(X_0,X_3),\ast),
  \]
  or {\em bracket}, to be the set of all secondary composites $\langle
  f_{23} \teth{} f_{12} \teth{} f_{01} \rangle$.
\end{defn}

\begin{prop}
  \label{prop:indeterminacymain}
  Changing the tethering and homotopy class of maps alters the value
  of a secondary composite by multiplication by loops, as follows. If
  $f_{23}$ is homotopic to $f_{23}'$, we have
  \[
  \langle f_{23} \teth{h} f_{12} \teth{k} f_{01} \rangle = 
  \langle f_{23}' \teth{h'} f_{12} \teth{k} f_{01}
  \rangle \cdot (u f_{01})
  \]
  for some $u \in \pi_1 \Map_\mathcal{C}(X_1,X_3)$ that is determined
  by $h$, $h'$, and a homotopy between $f_{23}$ and
  $f_{23}'$. Similarly, if $f_{01}$ is homotopic to $f_{01}'$, we have
  \[
  \langle f_{23} \teth{h} f_{12} \teth{k} f_{01} \rangle = 
  (f_{23} v)  \cdot \langle f_{23} \teth{h} f_{12} \teth{k'} f_{01}'  \rangle
  \]
  for some $v \in \pi_1 \Map_\mathcal{C}(X_0,X_2)$.

  If we replace all three maps with homotopic maps and choose new
  tetherings, we have
  \[
  \langle f_{23} \teth{h} f_{12} \teth{k} f_{01} \rangle = 
  (f_{23} v) \cdot \langle f_{23}' \teth{h'} f_{12}' \teth{k'} f_{01}'
  \rangle \cdot (u f_{01})
  \]
  for some $u \in \pi_1 \Map_\mathcal{C}(X_1,X_3)$ and $v \in \pi_1
  \Map_\mathcal{C}(X_0,X_2)$.
\end{prop}

In particular, this describes completely the indeterminacy in
secondary operations and brackets, and shows that (up to this
indeterminacy) a secondary operation or a bracket is well-defined on
homotopy classes of maps.

\begin{proof}
  We will prove the first identification, as the second is
  symmetric. Since $f_{23}$ and $f'_{23}$ are homotopic, there is a
  path $j\co f_{23} \Rightarrow f'_{23}$ in
  $\Map_{\mathcal{C}}(X_2,X_3)$. The composition
  \[
  jk\co \Delta^1 \times \Delta^1 \xrightarrow{j \times k} \Map_{\mathcal{C}}(X_2,
  X_3) \times \Map_{\mathcal{C}}(X_0,X_2) \to
  \Map_{\mathcal{C}}(X_0,X_3)
  \]
  determines a homotopy from $f_{23} k$ to $(j f_{12} f_{01})
  \cdot (f_{23}' k)$, making them equal in the fundamental groupoid of
  $\Map_{\mathcal{C}}(X_0,X_3)$.

  In this fundamental groupoid, we then have the following sequence of
  identities:
  \begin{align*}
  \langle f_{23} \teth{h} f_{12} \teth{k} f_{01} \rangle
    &= (f_{23} k)^{-1} \cdot (h f_{01})\\
    &= (f_{23}' k)^{-1} \cdot (j f_{12} f_{01})^{-1} \cdot (h f_{01})\\
    &= (f_{23}' k)^{-1} \cdot (h' f_{01}) \cdot (h' f_{01})^{-1} \cdot
      (j f_{12} f_{01})^{-1} \cdot (h f_{01})\\
    &= \langle f_{23}' \teth{h'} f_{12} \teth{k} f_{01} \rangle \cdot
      [(j f_{12} \cdot h')^{-1} \cdot h] f_{01}
  \end{align*}
  Letting $u = (j f_{12} \cdot h')^{-1} \cdot h \in \pi_1
  \Map_{\mathcal{C}}(X_0,X_2)$ gives the desired result.
\end{proof}

\begin{cor}
  \label{cor:secondaryindeterminacy}
  A secondary operation $\langle f_{23} \teth{h} f_{12}, -\rangle$ determines a
  well-defined map $\Phi$ on $\ker f_{12} \subset \pi_0
  \Map_\mathcal{C}(X_0,X_1)$ whose values are right cosets:
  \[
    \ker f_{12} \too{\Phi} (f_{23} \pi_1 \Map_\mathcal{C}(X_0, X_2))
    \Big \backslash \pi_1 \Map_{\cal C}(X_0, X_3).
  \]
  If two tetherings $h$, $h'$ give rise to operations $\Phi$,
  $\Phi'$, then there exists an element $u \in \pi_1 \Map_{\cal
    C}(X_1,X_3)$ such that
  \[
  \Phi x = \Phi' x \cdot (u x)
  \]
  for all $x \in \ker f_{12} \subset \pi_0 \Map_\mathcal{C}(X_0, X_1)$.

  Dual results hold for $\langle -, f_{12} \teth{} f_{01}\rangle$.
\end{cor}

\begin{cor}
  \label{cor:bracketindeterminacy}
  Suppose we have maps
  \[
  X_0 \too{f_{01}} X_1 \too{f_{12}} X_2 \too{f_{23}} X_3
  \]
  such that the double composites $f_{23} f_{12}$ and $f_{12} f_{01}$
  are nullhomotopic. Then the bracket
  $\langle f_{23},f_{12},f_{01} \rangle$ depends only on the homotopy
  classes of $f_{i,i+1}$ and is a well-defined double coset in
  \[
  (f_{23} \pi_1 \Map_\mathcal{C}(X_0,X_2)) \Big \backslash \pi_1 \Map_{\cal
    C}(X_0,X_3) \Big / (\pi_1 \Map_\mathcal{C}(X_1,X_3) f_{01}).
  \]
\end{cor}

\begin{rmk}
  \label{rmk:infinitybrackets}
  A more flexible version of the above constructions should exist,
  where basepoints are replaced by some appropriate system of maps
  $E_{i,j} \to \Map_\mathcal{C}(X_i, X_j)$ from contractible spaces
  $E_{i,j}$, together with appropriate lifts of the composition
  maps. For example, the category of diagrams of spaces $E \to X$ is a
  monoidal category under the pushout-product, and so we could ask for
  $\mathcal{C}$ to be enriched in this category with the constraint
  that the space $E$ is always contractible. We might instead try to
  find an appropriate analogue in terms of quasicategories satisfying
  certain basepoint conditions: in the notation of \cite{lurie-htt},
  the data to describe a secondary composite in
  Definition~\ref{def:secondarycomp} defines a map of enriched
  categories $\mf C[\partial \Delta^3] \to \mathcal{D}$, and the
  secondary composite is the obstruction to extending it to a map
  $\mf C[\Delta^3] \to \mathcal{D}$ (a homotopy coherent triple composite).
  Both of these constructions would apply more widely, but involve more
  bookkeeping and possibly require a more advanced technical
  framework. We have elected to use constructions in categories where
  this will not be necessary in order to minimize the technical load.
\end{rmk}

The definitions of secondary operations and brackets are preserved in
an obvious way under functors between enriched categories.
\begin{prop}
  \label{prop:functorialbrackets}
  Suppose $F\co \mathcal{C} \to \mathcal{C}'$ is an enriched functor
  between categories enriched in pointed spaces. Then any tethering
  $g \teth{h} f$ in $\mathcal{C}$ induces a tethering $Fg \teth{Fh}
  Ff$ in $\mathcal{C}'$. We have an equality
  \[
  F(\langle f_{23} \teth{h} f_{12} \teth{k} f_{01} \rangle) =
  \langle F f_{23} \teth{Fh} Ff_{12} \teth{Fk} Ff_{01} \rangle,
  \]
  and we have containments as follows:
  \begin{align*}
  F(\langle f_{23} \teth{h} f_{12}, f_{01} \rangle) &\subset
  \langle F f_{23} \teth{Fh} Ff_{12}, Ff_{01} \rangle\\
  F(\langle f_{23},  f_{12} \teth{k} f_{01} \rangle) &\subset
  \langle F f_{23}, Ff_{12} \teth{Fk} Ff_{01} \rangle\\
  F(\langle f_{23}, f_{12}, f_{01} \rangle) &\subset
  \langle F f_{23}, Ff_{12}, Ff_{01} \rangle\\
  \end{align*}
\end{prop}

There is a further extension in the case where we have an
\emph{enriched adjunction}. An example of such a result appears below.

\begin{prop}
\label{prop:adjunctionsecondary}
  Suppose that we have an enriched adjoints $F\co \mathcal{C} \to
  \mathcal{D}$ and $G\co \mathcal{D} \to \mathcal{C}$, encoded by a
  natural based homeomorphism
  \[
  \theta\co \Map_{\mathcal{C}}(X,GY) \cong \Map_{\mathcal{C}}(FX,Y).
  \]
  Given maps
  \[
  X_0 \too{f_{01}} X_1 \too{f_{12}} X_2 \too{g} G Y
  \]
  and tetherings
  \[
  g \teth{h} f_{12} \teth{k} f_{01},
  \]
  the map $\theta$ induces an identity
  \[
  \theta_* \langle g \teth{h} f_{12} \teth{k} f_{01}\rangle = \langle
  \theta g \teth{\theta h} F f_{12} \teth{Fk} F f_{01}\rangle.
  \]
\end{prop}

\begin{cor}
\label{cor:adjunctionbracket}
  There are containments
  \[
  \theta_* \langle g, f_{12} \teth{k} f_{01}\rangle \subset \langle
  \theta g, F f_{12} \teth{Fk} F f_{01}\rangle
  \]
  and
  \[
  \theta_* \langle g, f_{12}, f_{01}\rangle \subset \langle
  \theta g, F f_{12}, F f_{01}\rangle.
  \]
\end{cor}

\subsection{Pointings and augmentations}

In this section we let $\mathcal{D}$ be a category enriched in
spaces (now assumed to have no basepoint). In this section we indicate
a construction that replaces $\mathcal{D}$ with a category enriched in
pointed spaces.

\begin{defn}
  \label{def:ppa}
  An {\em augmented object} of $\mathcal{D}$ is an object $X \in
  \mathcal{D}$ equipped with a map $Y \to \emptyset$ to an initial
  object of $\mathcal{D}$. The space of maps between two augmented
  objects is the subspace of ordinary maps that commute with the
  augmentations.

  A {\em pointed object} of $\mathcal{D}$ is an object $Z \in
  \mathcal{D}$ equipped with a map $\ast \to Z$ from a terminal object
  of $\mathcal{D}$. The space of maps between two pointed objects is
  the subspace of ordinary maps that commute with the pointings.
\end{defn}

\begin{defn}
  Suppose $\mathcal{D}$ is a category enriched in spaces. We define
  $\ppa{\mathcal{D}}$, the category of {\em possibly pointed or
    augmented objects of $\mathcal{D}$}, to be the following category
  enriched in based spaces.

  An object of $\ppa{\mathcal{D}}$ is one of three types:
  \begin{enumerate}
  \item an augmented object $X \to \emptyset$ of $\mathcal{D}$,
  \item an ordinary object $Y$ of $\mathcal{D}$, or
  \item a pointed object $\ast \to Z$ of $\mathcal{D}$.
  \end{enumerate}
  The mapping spaces in $\ppa{\mathcal{D}}$ are given as
  follows.
  \begin{enumerate}
  \item The space of maps between two augmented objects $X \to
    \emptyset$, $X' \to \emptyset'$ is the space of maps of augmented
    objects, with basepoint given by the composite $X \to \emptyset \to X'$.
  \item The space of maps between two pointed objects $\ast \to Z$,
    $\ast' \to Z'$ is the space of maps of pointed objects, with
    basepoint given by the composite $Z \to \ast' \to Z'$.
  \item The space of maps between two ordinary objects $Y, Y'$ is the
    based space $\Map_{\mathcal{D}}(Y, Y')_+$, whose disjoint basepoint
    is called the {\em formal null map}.
  \item The space of maps from an augmented object $X \to \emptyset$
    to an ordinary object $Y$ is the space of maps $X \to Y$, with
    basepoint given by the map $X \to \emptyset \to Y$.
  \item The space of maps from an ordinary object $Y$ to a pointed
    object $\ast \to Z$ is the space of maps $Y \to Z$, with basepoint
    given by the map $Y \to \ast \to Z$.
  \item The space of maps from an augmented object $X \to \emptyset$
    to a pointed object $\ast \to Z$ is the space of maps $X \to Z$,
    with basepoint given by the canonical map factoring through either
    $\emptyset$ or $\ast$ in the commutative diagram
    \[
    \xymatrix{
      X \ar[r] \ar[d] & \emptyset \ar[dl] \ar[d] \\
      \ast \ar[r] & Z.
    }
    \]
  \item All other mapping spaces are one-point spaces---there are no
    non-basepoint maps from ordinary objects to augmented ones, or
    from pointed objects to ordinary ones. We also refer to these as
    {\em formal null maps}.
  \end{enumerate}
  We have full subcategories of $\ppa{\mathcal{D}}$ spanned by fewer
  than all three of these types of objects: for example, we have the
  categories of {\em augmented objects}, {\em pointed objects}, {\em
    possibly augmented objects}, and {\em possibly pointed
    objects} of $\mathcal{D}$.
\end{defn}

\begin{prop}
  \label{prop:ppapointed}
  The category $\ppa{\mathcal{D}}$ is enriched in pointed spaces under
  $\sma$.

  In $\ppa{\mathcal{D}}$, if a composite $X \to Y \to Z$ is
  nullhomotopic then $X$ is augmented, $Z$ is pointed, or
  one of the maps is a formal null map (in which case there is a
  canonical tethering).
\end{prop}

This construction makes it possible to take a category $\mathcal{D}$
and sensibly talk about secondary operations and brackets for a
composite $X_0 \to X_1 \to X_2 \to X_3$ in $\mathcal{D}$ if the first
map is a map of augmented objects, if the last map is a map of pointed
objects, or if the first object is augmented and the last object is
pointed. (If the maps arise from $\mathcal{D}$ then a formal null
map cannot appear.)

\begin{exam}
  \label{exam:hocolimbracket}
  If $\mathcal{C}$ has homotopy pushouts and we have augmented
  objects $X_0 \to X_1 \to \emptyset$, the bracket can be identified
  with an element in $\pi_0 \Map_\mathcal{C}(\Sigma X_0, X_3)$,
  represented by the outside rectangle in the homotopy coherent
  diagram
  \[
  \xymatrix{
  X_0 \ar[r] \ar[d] & X_1 \ar[r] \ar[d] & \emptyset \ar[d] \\
  \emptyset \ar[r] \ar@{}[ur]|-*[@]{\Rightarrow}
  & X_2 \ar[r] \ar@{}[ur]|-*[@]{\Rightarrow} & X_3.
  }
  \]
  The indeterminacy in the bracket is given by path concatenation with
  composites of either of the following forms:
  \[
  \xymatrix{
    \Sigma X_0 \ar[r]^v & X_2 \ar[r]^-{f_{23}} & X_3 &
    \Sigma X_0 \ar[r]^{\Sigma f_{01}} & \Sigma X_1 \ar[r]^-u & X_3
  }
  \]
  Dual results hold if we are given pointed objects $\ast \to X_2 \to
  X_3$, so that the bracket can be identified with an element in
  $\pi_0 \Map_\mathcal{C}(X_0, \Omega X_3)$. To avoid grief in these
  identifications, especially with respect to a loop-suspension
  adjunction, it is important to pay attention to the orientation of
  $S^1$ as detailed at length in
  \cite{harper-secondaryoperations}. This is why we have indicated
  directions for 2-cells.

  In the ``mixed'' case, there is little profound that we can say
  other than identification of a element in the bracket with the loop
  determined by a homotopy coherent diagram
  \[
  \xymatrix{
  X_0 \ar[r] \ar[d] & X_1 \ar[r] \ar[d] & \ast \ar[d] \\
  \emptyset \ar[r] \ar@{}[ur]|-*[@]{\Rightarrow}
  & X_2 \ar[r] \ar@{}[ur]|-*[@]{\Rightarrow} & X_3.
  }
  \]
\end{exam}

\subsection{Juggling and Peterson--Stein formulas}

In this section we return to assuming that we have a category
$\mathcal{C}$ enriched in based spaces.

There are several ``juggling'' formulas that describe the relationship
between brackets and function composition. All of them are obtained by
choosing representative nullhomotopies and composing them
appropriately, as in the Peterson--Stein formulas
\cite{peterson-stein-twoformulas}.
\begin{lem}
  \label{lem:tetheredps1}
  Suppose we have a sequence of objects $(X_0,\dots,X_4)$, together
  with maps $f_{i,i+1}\co X_i \to X_{i+1}$ and tetherings
  \[
  f_{34} \teth{} f_{23} \teth{} f_{12} \teth{} f_{01}.
  \]
  Then there is an identity
  \[
  f_{34}\langle f_{23} \teth{} f_{12} \teth{} f_{01}\rangle^{-1} =
  \langle f_{34} \teth{} f_{23} \teth{} f_{12} \rangle f_{01}
  \]
  in $\pi_1 \Map_\mathcal{C}(X_0, X_4)$.
\end{lem}

\begin{exam}
  \label{exam:tetheredps1}
  In the case where $X_2 \to X_3 \to X_4$ are maps of pointed objects in a
  category $\mathcal{D}$, this Peterson--Stein relation expresses that
  both loops in $\Map_\mathcal{\ppa{D}}(X_0,X_4)$ are homotopic to the
  loop determined by the following homotopy coherent diagram:
  \[
  \xymatrix{
  X_0 \ar[r] \ar[d]
  & \ast \ar[d] \ar@{=}[dr] \\
  X_1 \ar[d] \ar[r] \ar@{}[ur]|-*[@]{\Rightarrow}
  & X_2 \ar[r] \ar[d]
  & \ast \ar[d] \\
  \ast  \ar@{}[ur]|-*[@]{\Rightarrow} \ar[r]
  &X_3 \ar[r]  \ar@{}[ur]|-*[@]{\Rightarrow}
  &X_4
  }
  \]
  Similarly, in the mixed case we will need to derive Peterson--Stein
  relations from diagrams such as the following:
  \[
  \xymatrix{
  X_0 \ar[r] \ar[d]
  & \emptyset \ar[d] \ar[dr] \\
  X_1 \ar[d] \ar[r] \ar@{}[ur]|-*[@]{\Rightarrow}
  & X_2 \ar[r] \ar[d]
  & \ast \ar[d] \\
  \emptyset  \ar@{}[ur]|-*[@]{\Rightarrow} \ar[r]
  &X_3 \ar[r]  \ar@{}[ur]|-*[@]{\Rightarrow}
  &X_4
  }
  \]
\end{exam}

\begin{lem}
  \label{lem:tetheredps2}
  Each of the following juggling formulas holds whenever defined.
  \begin{align*}
    f_{34}\langle f_{23}\teth{h} f_{12} \teth{k} f_{01}\rangle &=
    \langle f_{34} f_{23} \teth{f_{34} h} f_{12} \teth{k} f_{01} \rangle\\
    \langle f_{34} f_{23}\teth{h} f_{12} \teth{k} f_{01}\rangle &=
    \langle f_{34} \teth{h} f_{23} f_{12} \teth{f_{23} k} f_{01} \rangle\\
    \langle f_{34} \teth{h f_{12}} f_{23} f_{12} \teth{k} f_{01}\rangle &=
    \langle f_{34} \teth{h} f_{23} \teth{k} f_{12} f_{01} \rangle\\
    \langle f_{34} \teth{h} f_{23}\teth{k f_{01}} f_{12} f_{01} \rangle &=
    \langle f_{34} \teth{h} f_{23}\teth{k} f_{12} \rangle f_{01}
  \end{align*}
\end{lem}

As we range over possible choices of tethering, these lemmas
expressing equality of secondary composites become containment
relations for secondary operations and brackets.

\begin{prop}
  \label{prop:secondaryjuggles}
  Each of the following juggling formulas for secondary operations
  holds whenever both sides are defined:
  \begin{align*}
    \langle f_4 \teth{} f_3, f_2\rangle f_1 &=  f_4 \langle f_3, f_2,
    f_1\rangle^{-1} \\
    f_4 \langle f_3 \teth{h} f_2, f_1\rangle &= \langle f_4f_3
    \teth{f_4 h} f_2, f_1\rangle\\
    \langle f_4f_3 \teth{h} f_2, f_1\rangle &\subset \langle f_4
    \teth{h} f_3f_2, f_1\rangle\\ 
    \langle f_4\teth{hf_2} f_3f_2, f_1\rangle &= \langle f_4
    \teth{h} f_3, f_2f_1\rangle\\
    \langle f_4\teth{h} f_3, f_2 f_1\rangle &\supset \langle f_4
    \teth{h} f_3, f_2\rangle f_1
  \end{align*}
  Dual results hold for secondary operations $\langle -, f_2 \teth{}
  f_1\rangle$. 
\end{prop}
\begin{proof}
  We will give the argument for the first statement, as the others are
  similar but less complex. Given fixed tetherings $f_4 \teth{h} f_3
  \teth{k} f_2 \teth{\ell} f_1$, we find that the left-hand side
  consists of elements of the following form:
  \begin{align*}
    \langle f_4 \teth{h} f_3\teth{k'} f_2 \rangle f_1
    &= \left[(f_4 v) \cdot \langle f_2 \teth{h} f_3 \teth{k}
      f_2\rangle \right] f_1\\
    &= (f_4 v f_1) \cdot \langle f_2 \teth{h} f_3 \teth{k}
      f_2\rangle f_1
  \end{align*}
  The right-hand side consists of elements of the following form:
  \begin{align*}
    f_4 \langle f_3 \teth{k'} f_2\teth{\ell'} f_1 \rangle^{-1}
    &= f_4 \left[(f_3 w) \cdot \langle f_2 \teth{h} f_3 \teth{k}
      f_2 \rangle \cdot (u f_1) \right]^{-1}\\
    &= f_4 \left[(u^{-1} f_1) \cdot \langle f_2 \teth{h} f_3 \teth{k}
      f_2 \rangle^{-1} \cdot (f_3 w^{-1})\right]\\
    &= (f_4 u^{-1} f_1) \cdot f_4 \langle f_2 \teth{h} f_3 \teth{k}
      f_2 \rangle^{-1} \cdot (f_4 f_3 w^{-1})
  \end{align*}
  However, $f_4 f_3 w^{-1}$ is always trivial because $f_4 f_3$ is
  nullhomotopic, and so the two sets coincide by
  Lemma~\ref{lem:tetheredps1}.
\end{proof}

\begin{prop}
  \label{prop:bracketjuggles}
  Each of the following juggling formulas for brackets holds whenever
  both sides are defined:
  \begin{align*}
    \langle f_4, f_3, f_2\rangle f_1  &= f_4 \langle f_3, f_2, f_1\rangle^{-1} \\
    f_4\langle f_3, f_2, f_1\rangle &\subset \langle f_4 f_3, f_2,
    f_1 \rangle\\
    \langle f_4 f_3, f_2, f_1\rangle &\subset \langle f_4, f_3 f_2,
    f_1 \rangle\\
    \langle f_4, f_3 f_2, f_1\rangle &\supset \langle f_4, f_3, f_2
    f_1 \rangle\\
    \langle f_4, f_3, f_2 f_1\rangle &\supset \langle f_4, f_3,
    f_2\rangle f_1
  \end{align*}
\end{prop}

We end with a remark on adjunctions. In the presence of an (enriched)
adjunction between categories $\mathcal{C}$ and $\mathcal{D}$, we can
describe relationships between secondary operations. Recall that an
enriched functor $F\co \mathcal{C} \to \mathcal{D}$ with enriched left
adjoint $G$ determines (and is determined by) an enriched category
$\mathcal{E}$ with object set $Ob(\mathcal{C}) \cup Ob(\mathcal{D})$,
such that:
\[
\Map_\mathcal{E}(x,y) = \begin{cases}
  \Map_\mathcal{C}(x,y) &\text{if }x,y \in \mathcal{C}\\
  \Map_\mathcal{D}(x,y) &\text{if }x,y \in \mathcal{D}\\
  \Map_\mathcal{C}(x,Gy) \cong \Map_\mathcal{D}(Fx,y) &\text{if }x \in \mathcal{C},
  y \in \mathcal{D}\\
  \emptyset &\text{otherwise}
\end{cases}
\]
This allows us to describe augmented and pointed objects in the
presence of an adjunction and define brackets even amongst objects in
categories related by adjunctions. We could, if desired, rephrase
several of our constructions in these terms, in particular with
respect to brackets that involve maps out of free objects.

\subsection{Additive structures}

In prominent examples, some of the mapping spaces in $\mathcal{C}$
have natural ``addition'' structures.

\begin{defn}\label{def:hspace}
  An object $Y \in \mathcal{C}$ is an {\em H-object} if
  $\Map_{\mathcal{C}}(-,Y)$ naturally takes values in $H$-spaces: it
  is equipped with a natural homotopy-unital binary operation $+$
  whose unit is the basepoint. A {\em map of H-objects} is a map $Y
  \to Y'$ preserving this structure.

  An object $X \in \mathcal{C}$ is an {\em co-H-object} if
  $\Map_{\mathcal{C}}(X,-)$ naturally takes values in $H$-spaces: it
  is equipped with a natural homotopy-unital binary operation $+$
  whose unit is the basepoint. A {\em map of co-H-objects} is a map $X
  \to X'$ preserving this structure.
\end{defn}

\begin{prop}
\label{prop:addtether}
  Suppose $X$ is a co-H-object in $\mathcal{C}$ and that we have maps
  $f, f'\co X \to Y$ and $g\co Y \to Z$, together with tetherings $g 
  \teth{h} f$ and $g \teth{h'} f'$. Then the pointwise product on
  paths in $\Map_{\mathcal{C}}(X,-)$ gives a tethering $g \teth{h +
    h'} (f+f')$.
\end{prop}

\begin{prop}
\label{prop:addbrackets}
  Each of the following addition formulas holds whenever both sides
  are defined and the source object is an co-H-object in
  $\mathcal{C}$:
  \begin{align*}
    \langle f_3 \teth{h} f_2 \teth{k+k'} (f_1+f_1')\rangle &=
    \langle f_3 \teth{h} f_2 \teth{k} f_1\rangle + 
    \langle f_3 \teth{h} f_2 \teth{k'} f_1'\rangle \\
    \langle f_3 \teth{h} f_2 , (f_1+f_1')\rangle &=
    \langle f_3 \teth{h} f_2 , f_1\rangle + 
    \langle f_3 \teth{h} f_2 , f_1'\rangle \\
    \langle f_3, f_2 , (f_1+f_1')\rangle &\subset
    \langle f_3, f_2 , f_1\rangle + 
    \langle f_3, f_2 , f_1'\rangle
  \end{align*}
  Dual results hold for $H$-objects.
\end{prop}
Here the addition on paths is the pointwise $H$-space structure. The
addition on $\pi_1 \Map_{\mathcal{C}}(X_0,X_3)$ is, by the
Eckmann--Hilton argument, equivalent to either path concatenation or
the pointwise $H$-space structure on paths, and makes this group
abelian.

\begin{proof}
  The first identity is expressed by the following interaction between
  path composition and the pointwise $H$-space structure:
  \begin{align*}
    [f_3 (k + k')]^{-1} \cdot [h (f_1 + f_1')]
    &= [(f_3 k)^{-1} + (f_3 k')^{-1}] \cdot [h f_1 + h f_1']\\
    &= [(f_3 k)^{-1} \cdot h f_1] + [(f_3 k')^{-1} \cdot h f_1']
  \end{align*}
  Letting $k$ and $k'$ vary over possible tetherings, this then shows
  that
  \[
  \langle f_3 \teth{h} f_2 , (f_1+f_1')\rangle \supset \langle f_3 \teth{h}
  f_2 , f_1\rangle + \langle f_3 \teth{h} f_2 , f_1'\rangle.
  \]
  The indeterminacy on the left-hand side consists precisely of adding
  elements of the form $f_3 u$, while on the right-hand side it
  consists of adding elements of the form
  $f_3 v + f_3 v' = f_3(v+v')$. Because the indeterminacy group is the
  same, this containment must be an equality of cosets.

  Now letting $h$ vary over possible tetherings (which produces a
  restricted set of elements on the right-hand side), we obtain the
  third identity.
\end{proof}

\subsection{Model categories}

Working in a model category often requires attention to
objects that are not cofibrant or fibrant, and function spaces for
such objects are poorly behaved. In this section we will spell out
adjustments to the construction of secondary operations which are more
convenient but equivalent to our standard construction.

Let $\mathcal{M}$ be a model category. Associated to this data there
is a \emph{hammock localization} $L^H\mathcal{M}$
\cite{dwyer-kan-calculatinglocalizations}. This is a simplicial
category with a functor $\mathcal{M} \to L^H\mathcal{M}$, bijective on
objects, that turns weak equivalences into homotopy equivalences. In
\cite{dwyer-kan-functioncomplexes} it is shown that $L^H\mathcal{M}$
recovers the homotopy theory of $\mathcal{M}$: it is invariant under
Quillen equivalence, the homotopy category of $L^H \mathcal{M}$ is
localization of $\mathcal{M}$ with respect to weak equivalences, and
if $\mathcal{M}$ is a simplicial model category there is a chain of
weak equivalences between $L^H \mathcal{M}$ and the simplicial
category of cofibrant-fibrant objects of $\mathcal{M}$.

With this in mind, for (possibly pointed or augmented) objects of
$\mathcal{M}$ it makes sense to calculate secondary composites
and brackets in either $\mathcal{M}$ or $L^H \mathcal{M}$. There are
natural maps
\[
\pi_k \Map_\mathcal{M}(X,Y) \to \pi_k \Map_\mathcal{M}(X_{cof},
Y_{fib}) \cong \pi_k \Map_{L^H\mathcal{M}} (X,Y),
\]
where the first is an isomorphism if $X$ is cofibrant and $Y$ is
fibrant. This natural map is compatible with function composition.

This means that a tethering, secondary composite, secondary operation,
or bracket in $\mathcal{M}$ determines a compatible one in
$L^H \mathcal{M}$. This use of $L^H \mathcal{M}$ then allows us to
discuss brackets, and identities between them, for maps in the
homotopy category of $\mathcal{M}$ without the inconvenience of using
cofibrant or fibrant replacements to obtain maps in
$\mathcal{M}$. When discussing secondary composites in $\mathcal{M}$,
we will regard this process as implicit.

\subsection{Secondary power operations}

The study of secondary operations can now be specialized to homotopy
operations for algebras over a fixed commutative ring spectrum $A$.

\begin{defn}
  \label{def:freealgebra}
  Given a commutative ring spectrum $A$, we let $\mb P_A^{E_n}$ be the
  left adjoint to the forgetful functor from $E_n$ $A$-algebras to
  spectra; if $n = \infty$ we simply write $\mb P_A$, and if $A = \mb
  S$ then we will omit $A$ from the notation.
\end{defn}

In particular, there is an isomorphism
\[
\mb P_A^{E_n}(X) \cong \bigvee A \sma \left(E_n(k)_+ \sma_{\Sigma_k} X^{\sma
  k}\right),
\]
where the spaces $E_n(k)$ are the terms in our chosen $E_n$-operad,
and the set of homotopy classes of maps of $E_n$ $A$-algebras $\mb
P_A^{E_n}(\vee S^{k_i}) \to C$ is naturally isomorphic to $\prod
\pi_{k_i} C$. The natural map $X \to \ast$ becomes a natural
augmentation $\mb P_A^{E_n}(X) \to A$, and a pinch map $X \to X \vee
X$ gives $\mb P_A^{E_n} (X)$ the structure of a co-H-object.

\begin{defn}
  \label{def:homotopyops}
  A {\em homotopy operation} on $E_n$ $A$-algebras is a natural
  transformation of functors
  \[
  \prod \pi_{k_i}(-) \to \pi_j(-),
  \]
  represented by a homotopy class of map of $E_n$ $A$-algebras
  \[
  \mb P_A^{E_n}(S^j) \to\mb P_A^{E_n}(\vee S^{k_i})
  \]
  or equivalently an element of
  \[
  \pi_j \mb P_A^{E_n}(\vee S^{k_i}) \cong \pi_j (A \sma \mb
  P^{E_n}(\vee S^{k_k}))
  \]
  If this operation preserves the zero element, we view it as
  determined by a map of augmented objects via the canonical
  projection to $A$; if it preserves addition, we view it as
  determined by a map of co-H-objects.

  Similarly, if $B$ is an $E_n$ $A$-algebra, a {\em homotopy operation
    on $E_n$ $A$-algebras under $B$} is a natural transformation in
  the homotopy category of $E_n$ $A$-algebras under $B$, represented
  by a homotopy class of map
  \[
  B \amalg \mb P_A^{E_n}(S^j) \to B \amalg \mb P_A^{E_n}(\vee S^{k_i}).
  \] 
  Here the coproduct $\amalg$ takes place in the category of
  $E_n$-algebras. If this operation preserves the zero element, we
  view it as determined by a map of augmented objects via the
  canonical projection to $B$; if it preserves addition, we view it as
  determined by a map of co-H-objects.
\end{defn}

Taking $B = A$ shows that the first type of operations are a special
case of the second, so there is no loss of generality in restricting
our attention to operations in the relative case. If $n = \infty$,
then conversely $E_\infty$ $A$-algebras under $B$ are equivalent to
$E_\infty$ $B$-algebras.

\begin{exam}
  \label{exam:scalarop}
  For any $b \in \pi_k(B)$ and any $n > 0$, multiplication by $b$
  determines an additive homotopy operation on $E_n$ $A$-algebras
  under $B$.
\end{exam}

\begin{rmk}
  \label{rmk:oppositecat}
  As above, the Yoneda lemma allows homotopy operations to be
  expressed as {\em pre}-composition with maps of free algebras. We
  usually write precomposition on the right, but this is at odds with
  the standard convention of writing operators (such as the
  Dyer--Lashof operations) on the left. We could attempt to solve this
  in many ways. One would be to work in an opposite category so that
  function application is on the right. One would be to notationally
  distinguish between maps between free algebras (operations), maps
  from free algebras to ordinary algebras (homotopy elements), and
  maps between ordinary algebras (maps). One is to accept the state of
  affairs, and resist the urge to use the same names for a
  Dyer--Lashof operation $Q^n$ and the map
  $\mb P_H(S^{j+n}) \to \mb P_H(S^j)$ that represents it. None of
  these solutions are good, but we have adopted the third because (in
  all honesty) it has confused us the least.
\end{rmk}

Relations between homotopy operations allow us to define secondary
operations in the following way.

\begin{defn}
  \label{def:secondarypower}
  Let $A$ be a commutative ring spectrum and $B$ an $E_n$
  $A$-algebra. Suppose we have homotopy operations $Q_i\co \prod_s
  \pi_{l_{i,s}} \to \pi_{k_i}$ and $R\co \prod_i \pi_{k_i} \to \pi_j$
  that preserve zero such that $R \circ (\prod_i Q_i) = 0$, realized
  by a homotopy coherent diagram
  \[
  \xymatrix{
    B \amalg \mb P_A^{E_n}(S^j) \ar[r]^-R \ar[d] &
    B \amalg \mb P_A^{E_n}(\vee_i S^{k_i}) \ar[d]^{Q}\\
    B \ar[r]  \ar@{}[ur]|-*[@]{\Rightarrow} &
    B \amalg \mb P_A^{E_n}(\vee_{i,s} S^{l_{i,s}})
  }
  \]
  of augmented $E_n$ $A$-algebras under $B$. We refer to $R$ as a {\em
    relation} between the operations $Q_i$. The coherence produces a
  tethering homotopy $h$, and the {\em secondary operation} associated to
  this relation is $\langle -, Q \teth{h} R\rangle$.
\end{defn}

\begin{prop}
  \label{prop:secondarypowerprops}
  Given $C$ any $E_n$ $A$-algebra under $B$, the domain of definition of the
  secondary operation $\langle -, Q \teth{h} R\rangle$ is the subset
  of $\prod \pi_{l_{i,s}} C$ of collections of elements $x_{i,s} \in
  \pi_{l_{i,s}} C$ such that $Q_i(x_{i,s}) = 0$ for all $i$. These are
  represented by homotopy commutative diagrams
  \[
  \xymatrix{
    B \amalg \mb P_A^{E_n}(\vee_i S^{k_i}) \ar[d]^{Q_i} \ar[r] & B \ar[d] \\
    B \amalg \mb P_A^{E_n}(\vee_{i,s} S^{l_{i,s}})\ar[r]_-{x_{i,s}}& C
  }
  \]
  of $E_n$ $A$-algebras under $B$. The value of $\langle -, Q \teth{h}
  R\rangle$ is a subset of $\pi_{j+1} C$, and the indeterminacy
  consists of adding elements in the image of the suspended operation
  $\sigma R\co \prod \pi_{k_i + 1} C \to \pi_{j+1} C$.
\end{prop}

\begin{prop}
  \label{prop:secondarypowernatural}
  Maps $f\co C \to D$ of $E_n$ $A$-algebras under $B$ preserve
  secondary operations.
\end{prop}

\begin{proof}
  This is the statement that
  \[
  f(\langle x, Q \teth{} R\rangle) \subset \langle f(x), Q \teth{}
  R\rangle,
  \]
  which is an application of the juggling formulas from
  Proposition~\ref{prop:secondaryjuggles}.
\end{proof}

\begin{rmk}
  If $n < m \leq \infty$, then the forgetful functors from $E_m$
  $A$-algebras under $B$ to $E_n$ $A$-algebras under $B$ also preserve
  secondary operations in the following sense. The forgetful functor
  $U$ from $E_m$ $A$-algebras under $B$ to $E_n$-algebras under $B$ has a
  left adjoint $F$, giving rise to an enriched adjunction. Since adjoints
  are preserved under composition, it preserves free objects:
  \[
  F\left(B \amalg \mb P_A^{E_n}(X)\right) \cong B \amalg \mb P_A^{E_m}(X).
  \]
  In particular, any homotopy operation
  \[
  Q\co B \amalg \mb P_A^{E_n}(S^{j}) \to B \amalg \mb
  P_A^{E_n}(\vee_i S^{k_{i}})
  \]
  for $E_n$ $A$-algebras under $B$ gives rise to a homotopy operation
  for $E_m$ $A$-algebras under $B$, defined by applying $FQ$ or,
  equivalently, by applying $U$ and then applying $Q$. By
  Corollary~\ref{cor:adjunctionbracket}, the enriched adjunction gives
  us canonical identifications
  \[
  \langle U-, Q \teth{h} R\rangle_{E_n} = \langle -, FQ \teth{Fh}
  FR\rangle_{E_m}
  \]
  showing that secondary operations are preserved by the forgetful
  functor.
\end{rmk}

We can also define functional homotopy operations as the analogues of
Steenrod's functional cohomology operations.
\begin{defn}
  \label{def:functionalpower}
  Suppose $A$ is a commutative ring spectrum and that we have maps $B
  \to C \too{f} D$ of $E_n$ $A$-algebras, making $f\co C \to D$ a map
  under $B$. Suppose that we have a homotopy operation $Q\co
  \prod_s \pi_{l_s} \to \pi_{k}$ for $E_n$ $A$-algebras under $B$ that
  preserves zero, realized by a commutative diagram
  \[
  \xymatrix{
    B \amalg \mb P_A^{E_n}(S^{k}) \ar[r]^-Q \ar[dr] &
    B \amalg \mb P_A^{E_n}(\vee_{s} S^{l_{s}}) \ar[d]\\
    & B.
  }
  \]
  The {\em functional homotopy operation} associated to this relation
  is the bracket $\langle f, -, Q\rangle$.
\end{defn}

\begin{prop}
  \label{prop:functionalpowerprops}
  For any maps of $E_n$ $A$-algebras $B \to C \too{f} D$, the domain of
  definition of the functional operation $\langle f, -, Q \rangle$ is
  the subset of $\prod \pi_{l_{s}} C$ of collections of elements
  $x_{s} \in \pi_{l_{s}} C$ such that $f(x_s) = 0$ and $Q(x_{s}) =
  0$. These are represented by homotopy commutative diagrams
  \[
  \xymatrix{
    B \amalg \mb P_A^{E_n}(S^{k}) \ar[r]^-Q \ar[d] &
    B \amalg \mb P_A^{E_n}(\vee_s S^{l_s}) \ar[d]_{x_s} \ar[r] & B \ar[d] \\
    B \ar[r] &
    C \ar[r]_-f&
    D
  }
  \]
  of $E_n$ $A$-algebras. The value of $\langle f, -, Q \rangle$
  is a subset of $\pi_{k+1} D$, and the indeterminacy consists of
  adding elements in the image of the suspended operation $\sigma Q\co
  \prod \pi_{l_s + 1} D \to \pi_{k+1} D$ and elements in the image of
  $f\co \pi_{k+1} C \to \pi_{k+1} D$.
\end{prop}

We now specialize the previous discussion to the category of
$E_n$-algebras over the mod-$2$ Eilenberg-Mac Lane spectrum $H$.
As in Example~\ref{exam:scalarop}, multiplication is one classical
example of a homotopy operation. Other examples of homotopy
operations, and relations between them, are furnished by {\em power
  operations}.
\begin{thm}[{\cite[III.3]{bmms-hinfty}}]
  \label{thm:bmmsops}
  For any commutative $H$-algebra $A$, there are homotopy
  operations 
  \[
  Q^s\co \pi_k \to \pi_{k+s}
  \]
  for $E_n$ $A$-algebras when $s < k+n-1$. These satisfy the following relations.
  \begin{enumerate}
  \item The {\em additivity relation}: $Q^s(x+y) = Q^s(x) + Q^s(y)$
  \item The {\em instability relations}: $Q^s x = x^2$ when $|x| = s$,
    $Q^s x = 0$ when $|x| > s$
  \item The {\em Cartan formula}: $Q^s(x y) = \sum_{p+q=s} Q^p(x)
    Q^q(y)$
  \item The {\em Adem relations}: If $r > 2s$, $Q^r Q^s (x) = \sum
    \binom{i-s-1}{2i-r} Q^{r+s-i} Q^i$ 
  \end{enumerate}
  For $m \leq n$, the forgetful map from $E_n$-algebras to
  $E_m$-algebras preserves Dyer--Lashof operations.

\end{thm}

\begin{prop}
  \label{prop:gotalloperations}
  For any commutative $H$-algebra $A$, all homotopy operations for
  $E_\infty$ $A$-algebras $C$ are composites of the following types:
  \begin{enumerate}
  \item the constant operation associated to an element $\alpha \in
    \pi_n A$, which takes no arguments and whose value on $C$ is the
    image of $\alpha$ under the map $\pi_* A \to \pi_* C$;
  \item the Dyer--Lashof operations $Q^s\co \pi_n(C) \to \pi_{n+s}(C)$;
  \item the binary addition operations $\pi_n(C) \times \pi_n(C) \to
    \pi_n(C)$;
  \item the binary multiplication operations $\pi_n(C) \times \pi_m(C)
    \to \pi_{n+m}(C)$.
  \end{enumerate}
\end{prop}

\begin{proof}
  The set of homotopy operations $\prod_s \pi_{l_s} \to \pi_*$ in this
  category is isomorphic to
  \[
    \pi_* (A \sma \mb P(\vee_s S^{l_s})) \cong \pi_* A \tens H_* \mb
    P(\vee_s S^{l_s}).
  \]
  Therefore, any homotopy operation is a sum of homotopy operations
  for $H$-algebras multiplied by constants from $\pi_* A$.  However,
  in \cite[IX.2.1]{bmms-hinfty} it is shown that the homology
  $H_* \mb P(X)$ is the free commutative algebra with Dyer--Lashof
  operations (subject to the additivity formula, instability
  relations, Cartan formula, and Adem relations) on $H_* X$, and so
  the homotopy operations for $H$-algebras are generated by constants,
  addition, multiplication, and the Dyer--Lashof operations $Q^s$.
\end{proof}

The category of $E_n$ $A$-algbras under $B$ has suspensions, and the
suspension of the augmented object $B \amalg \mb P_A^{E_n}(\vee_s
S^{l_s})$ is $B \amalg \mb P_A^{E_n}(\vee_s S^{l_s+1})$ 

\begin{prop}
  \label{prop:powerstability}
  The suspension operator $\sigma$, on homotopy operations for $E_n$
  $A$-algebras under $B$, takes zero-preserving homotopy operations
  $\prod \pi_{l_s} \to \pi_k$ to homotopy operations $\prod \pi_{l_s +
    1} \to \pi_{k+1}$. Suspension preserves addition, composition, and
  multiplication by scalars from $B$. Suspension also takes $Q^s$ to
  $Q^s$ and takes the binary multiplication operation $\pi_p \times
  \pi_q \to \pi_{p+q}$ to the trivial operation.
\end{prop}

\begin{rmk}
  \label{rmk:topopsucks}
  For $E_n$ $H$-algebras, there is also a ``top'' operation
  $\xi_{n-1}$ which, if $C$ extends to an $E_{n+1}$-algebra, agrees
  with to $Q^{k+n-1}$ on classes in $\pi_k C$. However, the top
  operation satisfies less tractable versions of the identities
  enjoyed by the remaining operations---most prominently, additivity
  requires correction by a new binary operation called the Browder
  bracket \cite[III.3.3]{bmms-hinfty}.
\end{rmk}

\subsection{Spectra and geometric realization}

For the following, we note that a tethering of a composite map of
spectra $X \too{f} Y \too{g} Z$ is equivalent to a homotopy class of
extension from the mapping cone $Cf$ to $Z$, up to orientation for the
interval component of the mapping cone.

\begin{prop}
  Suppose $X$, $Y$, and $Z$ are spectra, $X \too{f} Y \too{g} Z$ is
  nullhomotopic, and that $\alpha \in \ker(f) \subset \pi_n(X)$ is
  represented by a map $S^n \to X$. Given any extension
  $Y \to Cf \too{h} Z$ from the mapping cone representing a tethering,
  the secondary operation $\langle g \teth{} f, \alpha\rangle$ is (up
  to sign) the set $h(\partial^{-1} \alpha)$, where
  $\partial\co \pi_{n+1} Cf \to \pi_n X$ is the connecting
  homomorphism in the long exact sequence of homotopy groups.
\end{prop}

\begin{cor}
  \label{cor:geomreal}
  Suppose that $X_\star$ is a simplicial spectrum with geometric
  realization $|X_\star|$ and that $F$ is the homotopy fiber in the
  sequence $F \too{j} X_1 \too{d_0} X_0$. Then the composite
  $F \too{d_1j} X_0 \too{i} |X_\star|$ has a canonical tethering. If
  $\alpha \in \pi_n(F) \subset \pi_n X_1$ is in the kernel of $d_1$,
  then in the geometric realization spectral sequence
  \[
    H_p(\pi_q X_\star) \Rightarrow \pi_{p+q} |X_\star|
  \]
  the secondary operation $\langle i \teth{} d_1j, \alpha\rangle$ is
  represented (up to sign) by the element $[\alpha] \in H_1(\pi_n
  X_\star)$ in the spectral sequence.
\end{cor}

\begin{proof}
  The $1$-skeleton of $|X_\star|$, by definition, has a canonical
  diagram
  \[
    \xymatrix{
      X_1 \vee X_1 \ar[r] \ar[d]_{d_0 \vee d_1} & X_1 \sma [0,1]_+ \ar[d]\\
      X_0 \ar[r] & sk^{(1)} |X_\star|.
    }
  \]
  This defines a homotopy between the maps $i d_0$ and $i d_1$. The
  map $d_0 j$ has a canonical nullhomotopy by definition, and
  composing these two homotopies gives a canonical tethering
  \[
    i (d_1 j) \Rightarrow i (d_0 j) \Rightarrow *
  \]
  of $i (d_1 j)$. In particular, there is a canonical map
  $Cj \to sk^{(1)} |X_\star|$ from the mapping cone of $j$ to the
  $1$-skeleton of the geometric realization; by more carefully
  understanding the degeneracies, we can show that this map is
  a homotopy equivalence.

  By Proposition~\ref{prop:suspension}, in the resulting long exact
  sequence
  \[
    \dots \pi_{n+1} X_0 \too{i} \pi_{n+1} sk^{(1)} |X_\star| \too{\partial} \pi_n F
    \too{d_1j} \pi_n X_0 \too{i} \dots,
  \]
  any $\alpha \in \pi_n(F)$ which maps to zero under $d_1 j$ has a
  bracket $\langle i \teth{} d_1 j, \alpha\rangle$ in $\pi_{n+1}
  sk^{(n+1)} |X_\star|$, represented by any lift of $\alpha \in
  \pi_n F$, with indeterminacy given by the image of $i$.

  The spectral sequence for the homotopy groups of the geometric
  realization $|X_\star|$ is the spectral sequence associated to the
  following (unrolled) exact couple:
  \[
    \xymatrix{
      \ast \ar[r] & \pi_* sk^{(0)} |X_\star| \ar[r] \ar[d] & \pi_*
      sk^{(1)} |X_\star| \ar[r] \ar[d] & \dots \\
      & \pi_* sk^{(0)} \ar[ul] & \pi_* sk^{(1)} / sk^{(0)}  \ar[ul] &
    }
  \]
  Identifying the $0$-skeleton with $X_0$ and the next layer with the
  suspension of $F$, we obtain our desired identification of the
  element in the $E_1$-term with $\alpha$.
\end{proof}

We will now specialize to discuss how certain elements in a K\"unneth
spectral sequence can be identified with the results of secondary
operations.

\begin{prop}
  \label{prop:suspension}
  Suppose $f\co R \to S$ is a map of commutative ring spectra, and let
  $i = 1 \sma f\co S \sma R \to S \sma S$. Then, in the (pointed)
  category of augmented commutative $S$-algebras, there is a canonical
  tethering $p \teth{t} i$ for the composite
  \[
    S \sma R \too{i} S \sma S \too{p} S \sma_R S.
  \]

  Let $x \in \pi_n (S \sma R)$ map to zero in $\pi_n (S \sma S)$, so
  that $\sigma x = \langle p \teth{t} i, x\rangle \subset \pi_{n+1}S
  \sma_R S$ is defined. Then $\sigma x$ is detected by the image of
  $x$ under $\pi_n (S \sma R) \to \pi_n(S \sma R \sma S)$ in the
  two-sided bar construction spectral sequence
  \[
  H_p \pi_q (S \sma R^{\sma \star} \sma S) \Rightarrow \pi_{p+q} (S
  \sma_R S).
  \]
\end{prop}

\begin{proof}
  The relative smash product receives a map from the end of the augmented
  simplicial bar construction
  \[
    \xymatrix{
      S \sma R \sma S \ar@<-.5ex>[r] \ar@<.5ex>[r] & S \sma S \ar[r] &
      S \sma_R S,
    }
  \]
  a diagram of commutative ring spectra. The face maps
  \[
    d_j\co S \sma R \to S \sma R \sma S \to S \sma S
  \]
  are the null map $S \sma R \to S \too{\eta_L} S \sma S$ for $j=0$ and the
  map $S \sma R \too{i} S \sma S$ for $j=1$. Because the two composites $S
  \sma R \to S \sma_R S$ are homotopic, this provides a canonical
  tethering in the category of $S$-algebras.
  
  A homotopy element $x$ as described comes from a homotopy coherent
  diagram as follows:
  \[
    \xymatrix{
      S^n \ar[rr] \ar[dd]\ar[dr] && fib(d_1) \ar[r] \ar[d] & \ast
      \ar[d] \\
      & \mb P_S S^n \ar[r]^x \ar[d] \ar@{}[ur]|-*[@]{\Rightarrow} & S \sma R
      \ar[r] \ar[d]^{i} \ar@{}[ur]|-*[@]{\Rightarrow} & S \ar[d] \\
      \ast \ar[r] & S \ar[r] \ar@{}[ur]|-*[@]{\Rightarrow} &
      S \sma S \ar[r]_p \ar@{}[ur]|-*[@]{\Rightarrow} & S \sma_R S
    }
  \]
  The two lower right-hand squares define the bracket $\langle p, i,
  x\rangle$ in augmented commutative $S$-algebras, while the outside
  of the diagram is made up of two large (2-by-2 and 2-by-1)
  rectangles that are the result of forgetting down to
  spectra. However, by Corollary~\ref{cor:geomreal} the outside square
  determines an element of $\pi_{n+1} (S \sma_R S)$ which lifts to the
  desired element in the two-sided bar construction spectral sequence.
\end{proof}

\begin{rmk}
  The tethering plays an important role here. If we do not impose that
  the tethering $p \teth{t} i$ comes from a tethering in
  $E_\infty$-algebras, rather than spectra, then the indeterminacy for
  the bracket in spectra is too large to determine anything about
  bracket in $E_\infty$-algebras.
\end{rmk}

If $\pi_*(S \sma R)$ is flat over $\pi_* S$, we can identify the
$E_2$-term in the two-sided bar construction spectral sequence: 
\[
E^2_{**} = \Tor^{\pi_*(S \sma R)}_{**}(\pi_*(S \sma S), \pi_* S) \Rightarrow
\pi_*\left(S \sma_R S\right)
\]
The element $x$ gives rise to the corresponding element in
$\Tor_{1,n}$. In particular, we have the following result when the
target is the mod-$2$ Eilenberg--Mac Lane spectrum.

\begin{prop}
  \label{prop:geomdetection}
  Suppose $R \to H$ is a map of $E_\infty$-algebras and $x \in H_nR$
  maps to zero in the dual Steenrod algebra $H_* H$. Then there is an
  element $\sigma x = \langle p \teth{t} i, x\rangle$ in the $R$-dual
  Steenrod algebra $\pi_*(H \sma_R H)$ which is detected by the
  image of $x$ in homological filtration $1$ of the spectral sequence
  \[
  \Tor^{H_* R}_{**}(H_*, H_* H) \Rightarrow \pi_* \left(H \sma_R H\right)
  \]
\end{prop}

\begin{proof}
  In this case, we can rectify the map $R \to H$ to a weakly
  equivalent map between commutative ring spectra and apply
  Proposition~\ref{prop:suspension}.
\end{proof}

We now specialize this result to the case where $R$ is the complex
bordism spectrum.
\begin{prop}
\label{prop:mudetection}
  Let $n$ be an integer which is not of the form $2^k -1$ for any $k$,
  so that the corresponding generator $b_n \in H_{2n} MU \cong \mb
  F_2[b_1,b_2,\dots]$ in mod-$2$ homology is the Hurewicz image of the
  generator $x_n \in \pi_{2n} MU \cong \mb Z[x_1,x_2,\dots]$. Then the
  diagram of $E_\infty$ $H$-algebras
  \[
  \mb P_H(S^{2n}) \too{b_n} H \sma MU \too{p} H \sma H \too{i} H \sma_{MU} H,
  \]
  determines a bracket, and $\langle p, i, b_n\rangle \equiv \sigma x_n$
  mod decomposables.
\end{prop}

\begin{proof}
  The map $H_* MU \to H_* H$ is isomorphic to a map of polynomial algebras
  \[
  \mb F_2[b_1,b_2,\dots] \to \mb F_2[\xi_1,\xi_2,\dots]
  \]
  that sends $b_{2^k-1}$ to $\xi_k^2$ and sends the other generators to
  zero \cite[3.1.4]{ravenel-greenbook}. In particular, the K\"unneth
  spectral sequence
  \begin{equation}
  \Tor^{H_* MU}_{**}(H_*, H_* H) \Rightarrow \pi_* \left(H \sma_{MU}
    H\right)\label{eq:kunnethss}
  \end{equation}
  has as $E_2$-term an exterior algebra
  $\Lambda[\xi_k] \tens \Lambda[\sigma b_n \mid k \neq 2^{k}-1]$. By
  comparison with the K\"unneth spectral sequence
  \[
  \Tor^{\pi_* MU}_{**}(H_*, H_*) \Rightarrow \pi_* \left(H \sma_{MU} H\right),
  \]
  which degenerates and has $E_2$-term $\Lambda[\sigma x_k]$ of the
  same (graded) dimension, we find that spectral
  sequence~\eqref{eq:kunnethss} degenerates and that $\sigma b_n$ is
  congruent to $\sigma x_n$ mod decomposables for $n$ not of the form
  $2^k - 1$. We can then apply Proposition~\ref{prop:geomdetection} to
  identify $\sigma x_n$ as a secondary operation.
\end{proof}

\section{Hopf rings}
\label{sec:hopfrings}

\subsection{Background}

In this section we will recall some of the work of Ravenel--Wilson on
Hopf rings \cite{ravenel-wilson-hopfring}.

Let $E$ be a spectrum with a homotopy commutative multiplication and
let $\{E_n\}_{n \in \mb Z}$ be an associated $\Omega$-spectrum. Then
for any ring $R$ the homology groups $H_*(E_\star, R)$ have the
structure of a Hopf ring: they have a coproduct $\Delta$, an additive
product $\hash$, and a multiplicative product $\circ$ satisfying
associativity, commutativity, unitality, and distributivity laws that
make them into a graded ring object in coalgebras
\cite[1.12]{ravenel-wilson-hopfring}.\footnote{Ravenel--Wilson
  write $x \ast y$ for the additive product and $x \circ y$ for the
  multiplicative product, while Cohen--Lada--May
  \cite{cohen-lada-may-homology} write $xy$ for the additive product
  and $x \hash y$ for the multiplicative product.}
The constants $c \in E^{n} = \pi_0 E_{n}$ give rise to elements
$[c] \in H_0(E_n;R)$ under the Hurewicz map.

\begin{defn}
  \label{def:hopfgenerators}
  Suppose $E$ has a complex orientation $x \in \wt E^2(\mb{CP}^\infty)$
  realized by a based map $b\co \mb{CP}^\infty \to E_2$, and let $\beta_i
  \in H_{2i}(\mb{CP}^\infty;R)$ be dual to the generator $t^i \in
  H^{*}(\mb{CP}^\infty;R) \cong R[t]$. We define the classes $b_i \in
  H_{2i}(E_2;R)$ to be the images of $\beta_i$ under $f$.
\end{defn}

\begin{thm}[{\cite[4.6, 4.15, 4.20]{ravenel-wilson-hopfring}}]
  \label{thm:ravenelwilsonbasis}
  Let $\{MU_n\}$ be an $\Omega$-spectrum associated to complex
  cobordism. For any ring $R$ and any $n \in \mb Z$, $H_*(MU_{2n};R)$
  is, as an algebra under $\hash$, the tensor product of the group
  algebra $\mb Z[\pi_{-2n} MU]$ with a polynomial algebra over $R$.

  The even-degree indecomposables $Q^{\hash}\rH_*(MU_{2\star};R)$ under
  the $\hash$-product form a commutative graded ring under $\circ$, with
  relations as follows. If we define a formal power series $b(s) =
  \sum b_i s^i$ and write $x+_F y = \sum a_{i,j} x^i y^j$ for the
  formal group law of $MU_*$, then we have the Ravenel--Wilson
  relations
  \begin{equation}\label{eq:rwrelation}
    b(s+t) = \sum [a_{i,j}] \circ b(s)^{\circ i} \circ b(t)^{\circ j}.
  \end{equation}

  The ring $Q^{\hash} \rH_*(MU_{2\star}; R)$ is a quotient of the
  graded ring
  \[
  R[b_i] \otimes MU^{-2\star}
  \]
  by a regular sequence, determined by the Ravenel--Wilson
  relations. Both $Q^{\hash}\rH_*(MU_{2\star};R)$ and
  $H_*(MU_{2\star};R)$ are free over $R$.
\end{thm}

\begin{cor}
  \label{cor:hopfringcoefficients}
  For all $n$ and all primes $p$, we have commutative diagrams of the following form:
  \[
  \xymatrix{
    H_*(MU_{2n}; \mb Q) \ar[d] &
    H_*(MU_{2n}; \mb Z) \ar@{_`->}[l] \ar@{->>}[r] \ar[d]&
    H_*(MU_{2n}) \ar[d]\\
    H_{*-2n}(MU; \mb Q) &
    H_{*-2n}(MU; \mb Z) \ar@{_`->}[l] \ar[r]&
    H_{*-2n}(MU)\\
  }
  \]
\end{cor}

\subsection{The unstable homology invariant}

In the following, for spaces $X$ and $Y$ we will find it convenient to
identify $\Hom(H_* X, H_* Y)$ with the isomorphic completed tensor product
\[
H_*(Y) \ctens H^*(X).
\]
Here $H_*(Y)$ is discrete, while $H^*(X)$ inherits an inverse limit
structure dual to the filtration of $H_*(X)$ by finite-dimensional
subspaces.

The invariant below, in a slightly different form, appears as the
``total unstable operation'' in \cite[10.2]{goerss-dieudonne} and is
credited to Strickland.
\begin{defn}
  \label{def:unstableinvariant}
  Let $E$ be a multiplicative generalized cohomology theory
  represented by an $\Omega$-spectrum $\{E_n\}_{n \in \mb Z}$. The
  {\em unstable homology invariant} for $E$-cohomology is the
  collection of natural transformations of sets
  \[
  \Unstop\co E^n(X) = [X,E_n] \to \Hom(H_* X, H_* E_n) \cong H_* E_n
  \ctens H^*(X).
  \]
\end{defn}
\begin{rmk}
  For any $\alpha$, the element $\Unstop(\alpha) \in \Hom(H_* X, H_*
  E_n)$ is a coalgebra map that respects the Steenrod operations. This
  restriction will not be necessary for us to take into account here.
\end{rmk}

The groups $H_* E_n \ctens H^*(X)$ have products $\hash$ and $\circ$,
each individually induced by the corresponding product in the Hopf
ring and the cup product in $H^*(X)$. Using these, we can determine
how $\Unstop$ interacts with the ring structure in $E$-cohomology.
\begin{prop}
  \label{prop:unstableproperties}
  The unstable homology invariant $\Unstop$ satisfies the following
  formulas:
  \begin{align*}
    \Unstop(x+y) &= \Unstop(x) \hash \Unstop(y)\\
    \Unstop(xy) &= \Unstop(x) \circ \Unstop(y)\\
    \Unstop([c]) &= [c] \tens 1
  \end{align*}
  More specifically, for an element $z \in H_k(X)$ with coproduct
  $\Delta z = \sum z' \tens z''$, we have the identities
  \begin{align*}
    \Unstop(x+y)(z) &= \sum (\Unstop x)(z') \hash (\Unstop y)(z''),\\
    \Unstop(xy)(z) &= \sum (\Unstop x)(z') \circ (\Unstop y)(z'').
  \end{align*}
  For $z \in H_k(X)$ with augmentation $\epsilon(z) \in H_k(*)$ and
  $c \in E^{n}$, we have 
  \[
  \Unstop([c])(z) = \epsilon(z) [c].
  \]
\end{prop}

\begin{proof}
  Given elements $x, y \in E^n(X)$, represented by maps $X \to E_n$,
  the sum is represented by the composite
  \[
  \xymatrix{
    X \ar[r]^-\Delta & X \times X \ar[r]^-{(x,y)} & E_n \times E_n
    \ar[r]^-{\hash} & E_n.
  }
  \]
  Similarly, a product is represented by a composite
  \[
  \xymatrix{
    X \ar[r]^-\Delta & X \times X \ar[r]^-{(x,y)} & E_p \times E_q
    \ar[r]^-{\circ} & E_{p+q},
  }
  \]
  and a constant $c \in E^n$ by a composite
  \[
  X \to * \too{c} E_n.
  \]
  The desired identities follow by applying $H_*$. 
\end{proof}

\begin{rmk}
  \label{rmk:unstablepowerseries}
  In particular, for $X = \mb{CP}^\infty$ with mod-$p$ graded
  cohomology ring $\mb F_p[t]$, we can view the unstable homology
  invariant as a map
  \[
  E^n(\mb{CP^\infty}) \to H_*(E_n)\pow{t}.
  \]
  When $E$ is complex oriented, the orientation class $x \in
  E^2(\mb{CP}^\infty)$ is taken to the power series
  \[
  \Unstop(x) = \sum_{i \geq 0} b_i t^i \in H_*(E_2)\pow{t}
  \]
  (Definition~\ref{def:hopfgenerators}) denoted by $b(t)$ in
  \cite{ravenel-wilson-hopfring}. In these terms, Ravenel--Wilson's
  identity
  \[
  b(s+t) = b(s) +_{[F]} b(t) = \Hash_{i,j} [a_{i,j}] \circ b(s)^{\circ i} \circ b(t)^{\circ j}
  \]
  is proved by first applying $\Unstop$ to the identity $m^*(t) = \sum
  a_{i,j} s^i t^j$ in $E^2(\mb{CP}^\infty \times \mb{CP}^\infty)$ and
  then using naturality of $\Unstop$ to write $\Unstop m^*(t) = m^*
  b(t) = b(s+t)$.
\end{rmk}

While we will not require it, it can be clarifying to examine a
``reduced'' version of this invariant, especially in cases where $X$
has a basepoint. We begin by observing that $\Lambda(\alpha) - [0]$
takes values in reduced homology for any $\alpha \in E^*(X)$.

\begin{defn}
  \label{def:reducedunstable}
  Let $E$ be a multiplicative generalized cohomology theory
  represented by an $\Omega$-spectrum $\{E_n\}_{n \in \mb Z}$. The
  {\em reduced unstable homology invariant} for $E$-cohomology is the
  natural transformation of sets
  \[
  \unstop\co E^n(X) = [X,E_n] \to \Hom(H_* X, \rH_* E_n) \cong \rH_*
  (E_n) \ctens H^*(X)
  \]
  given by $\unstop(\alpha) = \Unstop(\alpha) - [0].$
\end{defn}

The identities for the operator $\Unstop$ translate into ones for
$\unstop$ which are particularly transparent when taken mod
decomposables for $\hash$.
\begin{prop}
  \label{prop:reducedproperties}
  The reduced unstable homology invariant $\unstop$ satisfies the
  following formulas:
  \begin{align*}
    \unstop(x+y) &= \unstop(x) + \unstop(y) + \unstop(x) \hash \unstop(y)\\
    \unstop(xy) &= \unstop(x) \circ \unstop(y) \\
    \unstop([c]) &= [c] - [0]
  \end{align*}
  The composite map
  \[
  E^\star(X) \too{\unstop} \rH_*(E_\star) \ctens H^*(X)
  \to (Q^{\hash}\rH_*(E_\star)) \ctens H^*(X)
  \]
  which reduces mod $\hash$-decomposables is a natural homomorphism of
  graded $E^\star$-algebras.
\end{prop}

Finally, we consider the case of reduced cohomology.

\begin{prop}
\label{prop:reducedtoreduced}
  Suppose $\alpha \in \widetilde E^n(X)$ corresponds to a based map $X
  \to E_n$. Then the reduced unstable invariant $\unstop(\alpha)$
  naturally takes values in $\rH_* E_n \ctens \rH^*(X)$.
\end{prop}

\begin{proof}
  There is a restriction map $\rH_* E_n \ctens H^*(X) \to \rH_* E_n
  \tens \mb F_p$ induced by the inclusion of the basepoint $* \to X$. An
  element $\alpha \in E^n(X)$ which restricts to an element $c \in
  E^n(*)$ at the basepoint is sent to the element $\unstop(\alpha) =
  \Unstop(\alpha) - [0]$ which restricts to $[c] - [0] \in \rH_*
  E_n$. If the map is based, then $c=0$ and so $\unstop(\alpha)$ lifts
  to the tensor with reduced cohomology.
\end{proof}

\subsection{Unit groups}

For a ring spectrum $E$, the space $SL_1(E) \subset \Omega^\infty E$
of strict units is the path component of the multiplicative unit $1
\in \pi_0(E)$. This construction is functorial in $E$. If we define
$\wt E_0 \subset E_0$ to be the path component of $0$, then there is a
homotopy equivalence $\wt E_0 \to SL_1(E)$ given by applying $[1] \#
(-)$. In particular, there are canonical isomorphisms $\pi_k(SL_1(E))
\cong \pi_k(E)$ for $k > 0$ and $H_k (SL_1(E)) \cong H_k(\wt E_0)$.
When $E$ is an $E_\infty$-algebra, the space of units inherits a
corresponding structure.
\begin{thm}[{\cite[IV.1.8]{may-quinn-ray-ringspectra}}]
  \label{thm:sl1multiplicative}
  For $E$ an $E_\infty$-algebra, the space $SL_1(E)$ has a
  natural structure of an infinite loop space such that the map
  \[
  \Sigma^\infty_+ SL_1(E) \to E
  \]
  is a natural map of $E_\infty$-algebras.
\end{thm}

\begin{prop}
  \label{prop:sl1exists}
  Suppose $E$ is an $E_\infty$-algebra, $HR$ is an Eilenberg-Mac
  Lane spectrum for a commutative ring $R$, and $E \to HR$ is a map of
  $E_\infty$-algebras. Then there is a natural suspension map
  \[
  \sigma\co SL_1(E) \to \Omega SL_1(HR \sma_E HR),
  \]
  of infinite loop spaces realizing, for $k > 0$, the natural map
  $\pi_k E \to \Tor^{E_*}_{1,k}(R, R)$ in the K\"unneth spectral
  sequence
  \[
    \Tor^{E_*}_{**}(R, R) \Rightarrow \pi_* \left(HR \sma_E HR\right)
  \]
  of \cite[IV.4.1]{ekmm}.
\end{prop}

\begin{proof}
  Since $SL_1$ only depends on connective covers, without loss of
  generality we can assume that $E$ is connective.  We consider the
  commutative diagram
  \[
  \xymatrix{
    E \ar[r] \ar[d] & HR \ar[d] \\
    HR \ar[r] & HR \sma_E HR.
  }
  \]
  We then apply $SL_1$ to this diagram. The space $SL_1(HR)$ is
  contractible, so the commutative diagram of infinite loop spaces
  \[
  \xymatrix{
    SL_1E \ar[r] \ar[d] & SL_1(HR) \ar[d] \\
    SL_1(HR) \ar[r] & SL_1(HR \sma_E HR)
  }
  \]
  determines (up to contractible indeterminacy) two nullhomotopies of
  the diagonal map as infinite loop space maps. Gluing these
  nullhomotopies together gives a map of infinite loop spaces
  \[
  SL_1(E) \to \Omega SL_1\left(HR \sma_E HR\right).
  \]

  To show compatibility with the K\"unneth spectral sequence, we begin
  by recalling its construction. Setting $HR = M_0$, we iteratively
  find fiber sequences $M_{i+1} \to F_i \to
  M_{i}$ of $E$-modules which are exact on homotopy groups, where $F_i
  \simeq \vee_\alpha \Sigma^{n_\alpha} E$ is a free graded $E$-module, and
  smash over $E$ with $HR$; the resulting long exact sequences
  assemble into an exact couple that calculates $\pi_*(HR \sma_E HR)$
  with $E_2$-term the desired $\Tor$-groups. In particular, we may
  choose the unit map $E \to HR$ as one of the factors in the map $F_0
  \to M_0$, which gives us a map $\Sigma^\infty_+ SL_1(E) \to E \to F_0$.

  Let $\beta\co S^k \to SL_1(E)$ represent an element in $\pi_k
  SL_1(E)$ for $k > 0$, and consider the diagram
  \[
  \xymatrix{
    & \Sigma^\infty SL_1(E) \ar[r] \ar[d] \ar@{.>}[ddl]&
    \Sigma^\infty_+ SL_1(E) \ar[r] \ar[d] &
    \Sigma^\infty_+ SL_1(HR) \simeq \mb S \ar[d] \\
    & M_1 \ar[r] \ar[d] & F_0 \ar[r] \ar[d] & H R \ar[d] \\
    \Omega (HR \sma_E HR) \ar[r] &
    HR \sma_E M_1 \ar[r] & HR \ar[r] & HR \sma_E HR,
  }
  \]
  whose rows are fiber sequences and where the dotted arrow is the map
  induced by the map $\sigma$. The composite map $S^k \to
  \Sigma^\infty SL_1(E) \to M_1$ represents the element $\beta \in
  \ker(\pi_k E \to \pi_k HR)$, and lifts to a map $S^k \to F_1$. The
  image in $\pi_k (HR \sma_E F_1)$ is the element corresponding to
  $\beta$ in $\Tor_{1,k}^{E_*}(R,R)$. However, this also coincides
  with the suspension of $\beta$ under the dotted arrow that uses the
  two nullhomotopies of $\Sigma^\infty SL_1(E) \to HR$.
\end{proof}

\begin{cor}
  \label{cor:adjointsuspension}
  For a ring $R$, there are {\em suspension} maps
  \[
  \sigma\co \rH_*(SL_1(E); R) \to H_{*+1}(BSL_1 E; R) \to \pi_{*+1}\left(HR
    \sma_E HR\right).
  \]
  These are natural in maps $E \to HR$ of $E_\infty$-algebras, and on
  the Hurewicz image of $\pi_* BSL_1(E)$ these are given by the
  suspension map. When $R = \mb F_2$, this map commutes with the
  Dyer--Lashof operations.
\end{cor}

\begin{proof}
  The map $SL_1(E) \to \Omega SL_1(HR \sma_E HR)$ is adjoint to a map
  $BSL_1(E) \to SL_1(HR \sma_E HR)$ of infinite loop spaces. We begin
  with the map of $E_\infty$-algebras
  \[
  \Sigma^\infty_+ BSL_1(E) \to \Sigma^\infty_+ SL_1(HR \sma_E HR) \to
  HR \sma_E HR.
  \]
  The adjunction between $E_\infty$-algebras and $E_\infty$
  $HR$-algebras (using the left unit $HR \to HR \sma_E HR$) then
  produces a natural map
  \[
  HR \sma BSL_1(E)_+ \to HR \sma_E HR
  \]
  of $E_\infty$ $HR$-algebras realizing our desired map. In
  particular, if $R = \mb F_2$ this map of $H$-algebras commutes
  with the Dyer--Lashof operations.
\end{proof}

\begin{cor}
  \label{cor:sl1detection}
  In the commutative diagram
  \[
  \xymatrix{
    \rH_*(SL_1(MU);\mb Q) \ar[d] &
    \rH_*(SL_1(MU);\mb Z) \ar@{_`->}[l] \ar@{->>}[r] \ar[d] &
    \rH_*(SL_1(MU)) \ar[d] \\
    \pi_{*+1}(H\mb Q \sma_{MU} H\mb Q) &
    \pi_{*+1}(H\mb Z \sma_{MU} H\mb Z)\ar@{_`->}[l] \ar[r] &
    \pi_{*+1}(H \sma_{MU} H), \\
  }
  \]
  where the vertical maps are suspensions, the left-hand horizontal
  arrows are injective and the right-hand top horizontal arrow is
  surjective. In particular, the suspension map in mod-$2$ homology is
  determined by the rational suspension map. In addition, the right-hand
  vertical map preserves the Dyer--Lashof operations. 
\end{cor}

\begin{proof}
  The injectivity and surjectivity of the top rows was shown in
  Corollary~\ref{cor:hopfringcoefficients}. The injectivity of the
  bottom-left map follows because the comparison map of K\"unneth
  spectral sequences
  \[
  \Tor_{**}^{\pi_* MU}(\mb Z, \mb Z) \to
  \Tor_{**}^{\pi_* MU}(\mb Q, \mb Q)
  \]
  becomes an inclusion of exterior algebras $\Lambda[\sigma x_i] \to
  \mb Q \tens \Lambda[\sigma x_i]$, and both spectral sequences
  degenerate at the $E_2$-term. Therefore, the map $\pi_*(H\mb Z
  \sma_{MU} H\mb Z) \to \pi_*(H\mb Q \sma_{MU} H\mb Q)$ is injective.
\end{proof}

We can now examine the properties of the rational suspension map by
using the rational Hopf ring.
\begin{prop}
  \label{prop:rationalcomparison}
  In the rational Hopf ring, the suspension map 
  \[
  \rH_*(SL_1(MU);\mb Q) \to H_{*+1}(H\mb Q \sma_{MU} H\mb Q),
  \]
  in terms of the Ravenel--Wilson basis, is a composite
  \[
  \rH_*(SL_1(MU);\mb Q) \twoheadrightarrow Q^\circ Q^{\hash} \rH_*(MU_0; \mb Q) /
  (b_2,b_3,\dots) \to \pi_{*+1}(H\mb Q \sma_{MU} H\mb Q)
  \]
  that kills $\hash$-decomposables, $\circ$-decomposables, and $b_i$ for
  $i > 1$, and sends any of the remaining basis elements $[\alpha]
  \circ b_1^{\circ s}$ to the suspension class $\sigma \alpha$ in
  the K\"unneth spectral sequence from
  Proposition~\ref{prop:geomdetection}.
\end{prop}

\begin{proof}
  There is a commutative diagram of $E_\infty$ rings over $H\mb Q$:
  \[
  \xymatrix{
  \Sigma^\infty_+ SL_1(MU) \ar[r] \ar[d] & H\mb Q \sma SL_1(MU)_+
  \ar[d] \\
  MU \ar[r] & H\mb Q \sma MU
  }
  \]
  Applying the natural map $\rH_*(SL_1(-);\mb Q) \to H\mb Q \sma_{(-)}
  H\mb Q$, we find that the suspension map
  \[
  \rH_*(SL_1(MU);\mb Q) \to \pi_{*+1}(H\mb Q \sma_{MU} H\mb Q) \cong
  \pi_{*+1}(H\mb Q \sma_{H\mb Q \sma MU} H\mb Q)
  \]
  can be computed as the composite
  \[
  \rH_*(SL_1(MU);\mb Q) \to H_*(MU;\mb Q) \to H_{*+1}(H\mb Q \sma_{MU}
  H\mb Q).
  \]
  The first map, under the isomorphism $[-1]\hash (-)\co \rH_*(SL_1(MU)) \cong
  \rH_*(\wt{MU}_0)$, sends $\hash$-decomposables to zero, carries
  $\circ$-products to products, and takes the elements $b_i$ for $i >
  1$ to $\circ$-decomposable elements $b_i \equiv [a_i] \circ b_1^{\circ
    i}$ due to the Ravenel--Wilson relation \eqref{eq:rwrelation}. The
  second is the suspension map $\sigma$, which carries
  $\circ$-decomposables to zero. The element $[\alpha] \circ
  b_1^{\circ s}$ is the Hurewicz image of $\alpha$ which, by
  definition, is carried to the suspension $\sigma \alpha$.
\end{proof}

Taking this together with Corollary~\ref{cor:sl1detection}, we find
the following.
\begin{cor}
  \label{cor:integralcomparison}
  The suspension map
  \[
  \rH_*(SL_1(MU)) \to \pi_{*+1}\left(H \sma_{MU} H\right)
  \]
  on mod-$2$ homology, in terms of the Ravenel--Wilson basis, is a
  composite
  \[
  \rH_*(SL_1(MU)) \twoheadrightarrow Q^\circ Q^{\hash} H_*(MU_0) /
  (b_2,b_3,\dots) \to \pi_{*+1}\left(H \sma_{MU} H\right)
  \]
  that kills $\hash$-decomposables, $\circ$-decomposables, and $b_i$ for
  $i > 1$, and sends any of the remaining elements $[\alpha]
  \circ b_1^{\circ s}$ in the Ravenel--Wilson basis to the suspension
  class $\sigma \alpha$ from the K\"unneth spectral sequence.
\end{cor}

\begin{prop}
  \label{prop:DLstability}
  The suspension map $\sigma$ on mod-$2$ homology commutes with
  Dyer--Lashof operations.
\end{prop}

\begin{proof}
  This map is the composite
  \[
  \rH_*(SL_1(MU)) \to H_{*+1}(BSL_1(MU)) \to
  \pi_{*+1}\left(H \sma_{MU} H\right).
  \]
  The Dyer--Lashof operations on the homology of infinite loop spaces
  are stable, and hence preserved by the first map; the compatibility
  of the second map is Corollary~\ref{cor:adjointsuspension}.
\end{proof}

\section{Power operations}

\subsection{Power operations in complex oriented theories}
\label{sec:cpxor}

In this section we will recall the work from \cite{bmms-hinfty} on
power operations in cohomology theories, and specifically results on
$H_\infty^2$-algebra structures from \cite[VIII]{bmms-hinfty}.

For an $E_\infty$ (and hence $H_\infty$) ring spectrum $E$, the
$E$-cohomology of a (based) space $X$ has natural {\em power
  operations} as follows. Fix $m > 0$ and write $D_m$ for the extended power
functor given by
\[
D_m(Y) = (Y^{\sma m})_{h\Sigma_m}.
\]
Representing an element $\alpha \in E^0(X)$ as a map $\alpha\co
\Sigma^\infty X \to E$, we form the commutative diagram
\[
\xymatrix{
& \Sigma^\infty D_m X \ar[r]^-{D_m \alpha}
\ar[dr]_{\mathcal{P}_m(\alpha)}
& D_m E \ar[d]\\
\Sigma^\infty X \sma (B\Sigma_m)_+ \ar[ur]^\Delta \ar[rr]_{P_m(\alpha)} && E,
}
\]
where the right-hand map is induced by the multiplicative structure on $E$. In
particular, this produces natural power operations:
\begin{align*}
  \mathcal{P}_m\co &\wt E^0(X) \to \wt E^0(D_m X)\\
  P_m\co &\wt E^0(X) \to \wt E^0(X \sma (B\Sigma_m)_+)
\end{align*}
These are multiplicative in an appropriate sense, and by replacing $X$
with $X_+$ we obtain compatible unbased versions:
\begin{align*}
  \mathcal{P}_m\co &E^0(X) \to E^0((X^m)_{h\Sigma_m})\\
  P_m\co &E^0(X) \to E^0(X \times B\Sigma_m)
\end{align*}

Outside degree $0$, we cannot draw conclusions which are as strong in
general. Given an element $\alpha \in \wt E^n(X)$ represented by a
map $\Sigma^\infty X \to E \sma S^n$, we can only define part of the
desired diagram:
\[
\xymatrix{
& \Sigma^\infty D_m X \ar[r]^-{D_m \alpha}
\ar@{.>}[dr]_{\mathcal{P}_m(\alpha)} & D_m(E \sma S^n) \ar@{.>}[d]^{?}\\
\Sigma^\infty X \sma (B\Sigma_m)_+ \ar[ur]\ar@{.>}[rr]_{P_m(\alpha)}
&& E \sma S^{nm}
}
\]
With extra structure on $E$ we can complete this diagram when $n$ is a
multiple of some fixed constant $d$: this is the case where $E$ is
$H_\infty^d$-algebra \cite[I.4]{bmms-hinfty}. An $H_\infty^d$-algebra
is an algebra equipped with explicit extra structure maps $D_{m}(E
\sma S^{dn}) \to E \sma S^{dmn}$, multiplicative and compatible across
$n$ and $m$. These allow us to obtain power operations:
\begin{align*}
  \mathcal{P}_m\co &\wt E^{dk}(X) \to \wt E^{dmk}(D_m X)\\
  P_m\co &\wt E^{dk}(X) \to \wt E^{dmk}(X \sma (B\Sigma_m)_+)
\end{align*}
These are multiplicative, and replacing $X$ with $X_+$ gives
compatible unbased versions:
\begin{align*}
  \mathcal{P}_m\co &E^{dk}(X) \to E^{dmk}((X^m)_{h\Sigma_m})\\
  P_m\co &E^{dk}(X) \to E^{dmk}(X \times B\Sigma_m).
\end{align*}

Cohomology is representable, so we may apply the Yoneda lemma. 
Restricting to the case where $m$ is a chosen prime $p$ and $d=2$,
we get the following.
\begin{thm}
  \label{thm:powerops}
  If $E$ is an $H_\infty^2$-algebra, there are natural based and
  unbased power operations for $n \in \mb Z$:
  \begin{align*}
    P\co &\wt E^{2n}(X) \to \wt E^{2pn}(X \sma (B\Sigma_p)_+)\\
    P\co &E^{2n}(X) \to E^{2pn}(X \times B\Sigma_p)
  \end{align*}
  These are universally represented by maps of based spaces $E_{2n}
  \sma (B\Sigma_p)_+ \to E_{2pn}$, and satisfy $P(x) P(y) = P(xy)$.
\end{thm}

For instance, the complex bordism spectrum $MU$ is an
$H_\infty^2$-algebra \cite[VIII.5.1]{bmms-hinfty}, giving us power
operations on even-degree classes previously studied by tom Dieck and
Quillen \cite{tomdieck-cobordismoperations,quillen-elementaryproofs}
that extend the power operations in degree zero. The spectrum $MU$,
which is complex oriented and has canonical Thom classes for complex
vector bundles, also has the special property that these operations are
compatible with the Thom isomorphism, as described by Quillen.
\begin{prop}[{\cite{quillen-elementaryproofs}}]
  \label{prop:thomcompatibility}
  For any complex vector bundle $\xi \to X$ of dimension $k$, write
  $t(\xi) \in MU^{2k}(Th(\xi))$ for the canonical Thom class of $\xi$
  and $e(\xi) \in MU^{2k}(X_+)$ for the Euler class.
  
  The based operation $\mathcal{P}_m$ preserves Thom classes: we have
  \[
  \mathcal{P}_m (t(\xi)) = t(D_m \xi),
  \]
  where $D_m \xi$ is the extended power bundle over
  $(X^m)_{h\Sigma_m}$. Restricting along the diagonal, we have
  \[
  P_m t(\xi) = e(\xi \boxtimes \overline\rho) t(\xi)
  \]
  where $\overline\rho$ is the bundle on $B\Sigma_m$ induced by the
  reduced permutation representation of $\Sigma_m$ and $\xi \boxtimes
  \overline\rho$ is the exterior tensor bundle on $X \times
  B\Sigma_m$. In particular, the Thom isomorphism fits into a
  commutative diagram
  \[
  \xymatrix{
    MU^{2n}(X) \ar[r]^-{P_m} \ar[d]_{t(\xi)} &
    MU^{2mn}(X \times B\Sigma_m) \ar[d]^{e(\xi \boxtimes \overline\rho) t(\xi)} \\
    \wt{MU}^{2(n+k)}(Th(\xi)) \ar[r]_-{P_m} &
    \wt{MU}^{2m(n +k)}(Th(\xi) \sma (B\Sigma_m)_+).
  }
  \]
\end{prop}

The cohomology of symmetric groups is closely related to formal group
law theory \cite{quillen-elementaryproofs}, and in particular the
effect of the power operation $P$ on the canonical first Chern class $x
\in \wt{MU}^2(\mb{CP}^\infty)$ was determined by Ando \cite{ando-isogenies}.
\begin{thm}
  \label{thm:quillenpower}
  The inclusion $C_p \into \Sigma_p$ induces inclusions:
  \begin{align*}
    MU^*(B\Sigma_p) &\into MU^*\pow{\alpha} /
    [p]_F(\alpha) \\
    MU^*(\mb{CP}^\infty \times B\Sigma_p) &\into MU^*\pow{x,\alpha} /
    [p]_F(\alpha)
  \end{align*}
  In these coordinates, the power operation $P$ satisfies
  $P(x) = \prod_{i=0}^{p-1} (x +_F [i]_F(\alpha))$.
\end{thm}

The power operations $P\co MU^{2k} \to MU^{2pk}(B\Sigma_p)$ are in
principle determined by these results, naturality, and
multiplicativity, and are closely related to the Lubin
isogeny in the theory of formal group laws. However, it has been
difficult to obtain closed-form expressions for these power
operations. We will require the following computation of
Johnson--Noel, using the fact that the generator $x_2$ of the complex
cobordism ring in dimension $4$ is $\mb{CP}^2$.
\begin{thm}[{\cite[6.3]{noel-johnson-ptypical}}]
  \label{thm:johnsonnoel}
  The polynomial generator $x_2$ of the complex bordism ring $MU_* \cong \mb Z[x_1,
  x_2, x_3,\dots ]$, appearing in $\pi_4(MU)$, has image
  \[
  P(x_2) \equiv \alpha^2 (v_1^6 + v_2^2) + \alpha^3 (v_1^7 + v_3)
  \]
  in $BP_*\pow{\alpha} / ([2]_F(\alpha), \alpha^8)$, where $P$ is
  the $2$-primary power operation.  In particular, $P(x_2)
  \equiv x_7 \alpha^3$ mod decomposables and higher order terms in
  $\alpha$.
\end{thm}

\begin{rmk}
  The powers of $\alpha$ appearing in the above result differ by a
  shift from those in \cite{noel-johnson-ptypical} because their
  identity occurs after multiplication by a power of an Euler
  class.
\end{rmk}

The main result of this paper hinges on this theorem. In
Appendix~\ref{sec:powerops} we will show that Johnson--Noel's method
can be adapted to one that works in torsion-free quotients of the
Lazard ring. This tweak allows us to give an abbreviated version of
their proof at the prime $2$, ignoring decomposables, that is easier
to carry out without computer assistance.

\subsection{Unstable Dyer--Lashof operations}

We recall the computation of the cohomology of the symmetric group
$\Sigma_p$:
\[
H^*(B\Sigma_p) \cong
\begin{cases}
  \mb F_2[u] & \text{if }p=2,\\
  \mb F_p[u] \otimes \Lambda[v] &\text{if }p >2.
\end{cases}
\]
Here $u$ has degree $1$ if $p=2$, while $u$ has degree $2p-2$ and $v$ has
degree $2p-3$ if $p$ is odd. 

\begin{defn}
  \label{def:homologypowerops}
  If $E$ is an $H_\infty^2$-algebra, the {\em homology power operation}
  \[
  \mathcal{Q}\co H_* (E_{2n}) \to H_*(E_{2pn}) \ctens
  H^*(B\Sigma_p)
  \]
  is adjoint to the map
  \[
  H_* P\co H_* (E_{2n}) \tens H_* (B\Sigma_p) \to H_*(E_{2pn})
  \]
  induced by the map $E_{2n} \sma (B\Sigma_p)_+ \to E_{2pn}$ of
  based spaces from Theorem~\ref{thm:powerops}. 
\end{defn}

The multiplicativity of the natural power operation $P$ has the
following consequence.
\begin{prop}
\label{prop:homologymultiplicative}
  The operation $\mathcal{Q}$ satisfies $\mathcal{Q}(x) \circ
  \mathcal{Q}(y) = \mathcal{Q}(x \circ y)$ and $\mathcal{Q}([0]) =
  [0]$.
\end{prop}

\begin{prop}
  \label{prop:powercompatibility}
  Suppose $E$ is an $H_\infty^2$-algebra. Then for all $n \in \mb Z$ we have a
  commutative diagram of sets
  \[
  \xymatrix{
  \wt{E}^{2n}(X) \ar[d]_\Unstop \ar[r]^-{P} &
  \wt{E}^{2pn}(X \sma (B\Sigma_p)_+) \ar[d]^\Unstop \\
  H_* (E_{2n}) \ctens H^* (X) \ar[r]_-{\mathcal{Q} \tens 1} &
  H_* (E_{2pn}) \ctens H^* (B\Sigma_p) \ctens
  H^*(X)
  }
  \]
  that is natural in $X$. The horizontal maps preserve products
  and the bottom map is a map of abelian groups.
\end{prop}

\begin{proof}
  The power operation $P$ sends an element represented by a map
  $\alpha\co X \to E_{2n}$ to the composite
  \[
  P(\alpha)\co X \sma (B\Sigma_p)_+ \too{\alpha \sma 1} E_{2n} \sma
  (B\Sigma_p)_+ \too{P} E_{2pn}.
  \]
  The value of $\Unstop(P(\alpha))$ is the effect on homology, which
  is the composite
  \[
  H_*(X) \tens H_* (B\Sigma_p) \too{H_* \alpha \otimes 1}
  H_*(E_{2n}) \tens H_* (B\Sigma_p) \too{H_* P} H_*(E_{2pn}).
  \]
  Taking adjoints recovers the statement about completed tensor
  products.
\end{proof}

\begin{rmk}
  \label{rmk:unstablecyclic}
  The map $BC_p \to \mb{CP}^\infty$ induces a map
  $\wt{E}^2(\mb{CP}^\infty) \to \wt{E}^2(BC_p)$ that takes the
  orientation class $x$ to the generator $\alpha$ described in
  Theorem~\ref{thm:quillenpower}, and the map
  $H^*(\mb{CP}^\infty) \to H^*(BC_p)$ is the ring map that sends $t$
  to $u^2$ if $p$ is $2$ or to a generator $w = u^{1/(p-1)}$ in degree
  $2$ if $p$ is odd. By naturality of $\Unstop$, we find that
  $\Unstop(\alpha)$ is equal to $b(u^2)$ if $p=2$ and is equal to
  $b(w)$ if $p$ is odd.
\end{rmk}



For the remainder of this section we will focus on the prime $2$. We
first recall the following calculation, which is dual to the
identity used to define the Steenrod operations in \cite[VII.3.2,
VII.6.1]{steenrod-epstein}.
\begin{lem}
  For a space $X$ with second extended power $D_2(X)$, the composite
  diagonal map
  \[
  H_*(X) \otimes H_*(B\Sigma_2) \to H_*(X \times B\Sigma_2) \to H_*(D_2(X))
  \]
  on mod-$2$ homology is given by
  \[
  v \tens \beta_n \mapsto \sum_{j \geq 0} Q^{j+n} (P_j v).
  \]
  Here $\beta_n$ is dual to $u^n$ and $P_j$ is the homology operation dual to
  $Sq^j$.
\end{lem}
As a result, the Dyer--Lashof operations can be recovered from this
diagonal map into the extended power.

\begin{thm}
  \label{thm:homologypowerops2}
  Consider the homology operations \[
  \mathcal{Q}\co H_* (MU_{2n}) \to H_*(MU_{4n}) \ctens H^*(B\Sigma_2)
  \]
  from Definition~\ref{def:homologypowerops}. Then there are
  multiplicative Dyer--Lashof operations
  \[
  \mQ^s\co H_*(MU_{2n}) \to H_{*}(MU_{4n}),
  \]
  extending the Dyer--Lashof operations in degree zero of
  \cite[II.1]{cohen-lada-may-homology} (coming from the multiplicative
  $E_\infty$-space structure) to Dyer--Lashof operations in even
  degrees. These satisfy the Cartan formula
  \[
  \mQ^s(x \circ y) = \sum_{p+q=s} \mQ^p(x) \circ \mQ^q(y)
  \]
  and are related to $\mathcal{Q}$ by the identity  
  \[
  \mathcal{Q}(x) =
    \sum_{n,j} \mQ^{j+n}(P_jx) u^n.
  \]
\end{thm}

In particular, if all nontrivial Steenrod operations vanish on $x$
then $\mathcal{Q}(x) = \sum \mQ^{n}(x) u^n$. This property is
invariant under the product $\circ$.

\subsection{Power operations in the Hopf ring}

We can now begin to use the results of the previous sections to
calculate multiplicative Dyer--Lashof operations in the Hopf ring for
$MU$ (the additive ones having been determined by
Turner~\cite{turner-dyerlashof}). First we will find the effect on the
class $b_1 \in H_2(MU_2)$ of Definition~\ref{def:hopfgenerators},
because $\circ$-multiplication by $b_1$ represents suspension.

\begin{prop}[{cf. \cite{priddy-dyerlashof}}]
  \label{prop:orientationpowers}
  Let $b_k \in H_{2k}(MU_2)$ denote the fundamental classes of
  Definition~\ref{def:hopfgenerators}. Then the $2$-primary
  multiplicative Dyer--Lashof operations satisfy
  \[
  \mQ^{2n} b_1 = b_1 \circ b_n
  \]
  for all $n \geq 1$.
\end{prop}

\begin{proof}
  For a general prime $p$, we consider the commutative diagram
  \[
  \xymatrix{
    MU^0(BU(1)) \ar[r]^-P \ar[d]_{t(\gamma_1)} & MU^0(BU(1) \times
    B\Sigma_p) \ar[d]^{e(\gamma_1 \boxtimes \overline \rho) t(\gamma_1)}\\
    \wt{MU}^2(MU(1)) \ar[r]^-P \ar[d]_\Unstop &
    \wt{MU}^{2p}(MU(1) \sma (B\Sigma_p)_+) \ar[d]^\Unstop \\
    H_*(MU_2)\tens H^*(MU(1)) \ar[r]_-{\mathcal{Q} \tens 1} &
    H_*(MU_{2p}) \tens H^*(B\Sigma_p) \tens H^*(MU(1)),
  }
  \]
  where the top square expressing compatibility of $P$ with the Thom
  isomorphism is from
  Proposition~\ref{prop:thomcompatibility}. Because $x$ is the Thom
  class of the canonical bundle on $BU(1)$, $\Unstop(t(\gamma_1)) =
  b(s)$. The image of the unit $1 \in MU^0(BU(1))$ along the
  left-to-bottom composite is then
  \[
  (\mathcal{Q} \tens 1)(\Unstop(x)) = (\mathcal{Q} \tens 1)(b(s)) = \sum \mathcal{Q}(b_k) s^k.
  \]
  On the other hand, the image along the top-right composite is
  \[
  \Unstop\left(x \prod_{k=1}^{p-1}(x +_F [k]_F\alpha)\right)
  = b(s) \circ (b(s) +_{[F]} b(u^2)) \circ \dots \circ (b(s) +_{[F]}
  [p-1]_{[F]} b(u^2)),
  \]
  using the expression for the Euler class of the exterior tensor
  bundle $\gamma_1 \boxtimes \overline\rho$ on $BU(1) \times B\Sigma_p$.

  Taking the coefficient of $s$, which involves only the linear term
  of $b(s)$ and the constant coefficients (in terms of $s$) of the
  factors $b(s) +_{[F]} [k]_{[F]} b(u^2)$, we find that
  \[
  \mathcal{Q}(b_1) = b_1 \circ b(u^2) \circ [2]_{[F]} b(u^2) \circ
  \dots \circ [p-1]_{[F]} b(u^2).
  \]
  When we specialize to $p=2$ and apply Theorem~\ref{thm:homologypowerops2}, we find
  \[
  \sum_{j \geq 2} \mQ^j(b_1) u^j = b_1 \circ b(u^2) = \sum_{n \geq 1} (b_1 \circ b_n) u^{2n}
  \]
  as desired.
\end{proof}

\begin{prop}
  \label{prop:seriesexpression}
  Suppose that $y \in \pi_{2n} MU$ and that, in the coordinates of
  Theorem~\ref{thm:quillenpower}, we have
  \[
  P(y) = \sum_{i=0}^\infty c_i \alpha^i
  \]
  for some elements $c_i \in \pi_{4n+2i} MU$. Then
  \[
  \mathcal{Q}([y]) = \Hash_{i=0}^\infty [c_i] \circ b(u^2)^{\circ i}
  \]
\end{prop}

\begin{proof}
  Taking $X = *$ in Proposition~\ref{prop:powercompatibility}
  identifying $[y]$ with $\Unstop(y)$, we find
  \begin{align*}
    \mathcal{Q}([y]) &= \Unstop(P(y))\\
    &= \Unstop \left(\sum c_i \alpha^i\right)\\
    &= \Hash_{i=0}^\infty [c_i] \circ b(u^2)^{\circ i}
  \end{align*}
  by Proposition~\ref{prop:unstableproperties} and
  Remark~\ref{rmk:unstablecyclic}.
\end{proof}

\begin{cor}
  \label{cor:poweronhurewicz}
  Mod $\hash$-decomposables and the $\circ$-ideal generated by $b_2,
  b_3, \dots$, the Hurewicz image $[x] \circ b_1^{\circ n} \in H_{2n}
  (MU_0)$ of $x \in \pi_{2n} (MU_0)$ satisfies
  \[
  \mathcal{Q}([x]\circ b_1^{\circ n}) \equiv \sum_{i=0}^\infty [c_i] \circ
  (b_1)^{\circ (i+2n)} u^{2(i+n)}.
  \]
  In particular, $\mQ^{2k}([x] \circ b_1^{\circ n}) \equiv [c_{k-n}] \circ
  b_1^{\circ (k+n)}$ in this quotient.
\end{cor}

\begin{proof}
  The first part follows from the multiplication formula
  $\mathcal{Q}([x] \circ  b_1^{\circ n}) = \mathcal{Q}([x]) \circ
  \mathcal{Q}(b_1)^{\circ n}$. The second part follows from
  Theorem~\ref{thm:homologypowerops2} and the fact that the
  operations $P_j$ vanish on $[x] \circ (b_1)^{\circ n}$ for $j > 0$.
\end{proof}

\subsection{Power operations in the $MU$-dual Steenrod algebra}
We will now apply the previous technology to compute a multiplicative
Dyer--Lashof operation in $H_* SL_1(MU)$. In order to do so, we need
some preliminary results about how the additive product interacts with
multiplicative Dyer--Lashof operations.
\begin{prop}
  \label{prop:dlproperties}
  At $p=2$, the multiplicative and additive Dyer--Lashof operations
  in the Hopf ring of an $E_\infty$-algebra satisfy the following
  identities.
  \begin{enumerate}
  \item When $x$ and $y$ are in the positive-degree homology of the
    path component of zero, we have
    \[
    \mQ^s(x \hash y) \equiv Q^s(x \circ y)
    \]
    mod $\hash$-decomposables.
  \item When $y$ is in the positive-degree homology of the path
    component of zero, we have
    \[
    \mQ^s([1] \hash y) \equiv [1] \hash Q^s(y) + [1] \hash \mQ^s(y)
    \]
    mod $\hash$-decomposables.
  \item For any positive-degree element $x$ there exist elements $z_i$
    for $0 < i < |x|$ such
    that
    \[
    Q^s(x) = Q^s[1] \circ x  + \sum Q^{s+i}[1] \circ z_i.
    \]
    In particular, $Q^s(x)$ is $\circ$-decomposable for any $x$ and
    any $s > 0$.
  \end{enumerate}
\end{prop}

\begin{proof}
  The mixed Cartan formula \cite[II.2.5]{cohen-lada-may-homology}
  takes the following form. If $x$ and $y$ are elements with
  coproducts given by $\Delta x = \sum x' \otimes x''$ and
  $\Delta y = \sum y' \otimes y''$, then
  \[
  \mQ^s(x\hash y) = \sum_{p+q+r=s} \sum \mQ^p(x') \hash Q^q(x''
  \circ y') \hash \mQ^r(y'').
  \]
  In the case of the first identity, the only time this is not
  decomposable under $\hash$ is when both $\mQ^p(x')$ and $\mQ^r(y'')$
  are of degree zero; this occurs when $p=r=0$ and we take the terms
  $[0]\tens x$ and $y \tens [0]$ of the coproduct.

  In the case of the second identity, the only nonzero terms in the
  mixed Cartan formula occur when $p=0$ and either $y' = [0]$ or
  $y'' = [0]$.

  The third identity is proven by induction on the degree of $x$,
  using the formula
  \[
  Q^s([1]) \circ x = \sum Q^{s+i}([1] \circ P_i x)
  \]
  from \cite[II.1.6]{cohen-lada-may-homology}.
\end{proof}

\begin{cor}
  \label{cor:poweraftertranslation}
  When $x$ is in the positive-degree homology of the path
  component of zero, we have
  \[
    \mQ^s([1] \hash x) \equiv \mQ^s(x)
  \]
  mod $\hash$-decomposables and $\circ$-decomposables, and hence
  \[
    \mathcal{Q}([1] \hash x) \equiv \mathcal{Q}(x).
  \]
\end{cor}

\begin{proof}
  We have
  \begin{align*}
    \mQ^s([1] \hash x) - \mQ^s(x)
    &\equiv [1] \hash Q^s(x) + [1] \hash \mQ^s(x) - [0] \hash \mQ^s(x) \\
    &= ([1] - [0]) \hash (Q^s(x) + \mQ^s(x)) + Q^s(x)\\
    &\equiv 0
  \end{align*}
  because the first element is $\hash$-decomposable and the second is
  $\circ$-decomposable.
\end{proof}

\begin{prop}
  \label{prop:finalcomp}
  Suppose that $x \in \pi_{2n} MU$ and that, in the coordinates of
  Theorem~\ref{thm:quillenpower}, we have
  \[
  P(x) = \sum_{i=0}^\infty c_i \alpha^i
  \]
  for some elements $c_i \in \pi_{4n+2i} MU$. Then mod
  $\hash$-decomposables, $\circ$-decomposables, and the ideal generated
  by $b_2, b_3, \dots$, the Hurewicz image $[1] \hash ([x] \circ
  b_1^{\circ n})$ of $x \in \pi_{2n} SL_1 MU$ satisfies
  \[
  \mathcal{Q}([1] \hash ([x]\circ b_1^{\circ n})) \equiv \sum_{i=0}^\infty [c_i] \circ
  (b_1)^{\circ (i+2n)} u^{2(i+n)}.
  \]
  In particular, $\mQ^{2k}([1] \hash ([x] \circ b_1^{\circ n})) \equiv [c_{k-n}]
  \circ b_1^{\circ (k+n)}$ in this quotient.
\end{prop}

\begin{rmk}
  We are working mod $\circ$-decomposables in $H_*(MU_0)$,
  and not in the entire Hopf ring, and so the right-hand side is not
  necessarily $\circ$-decomposable unless $c_{k-n}$ is.
\end{rmk}

\begin{proof}
  By Corollary~\ref{cor:poweraftertranslation}, we have
  \[
    \mathcal{Q}([1] \hash ([x] \circ b_1^{\circ n})) \equiv
    \mathcal{Q}([x] \circ b_1^{\circ n}),
  \]
  and by Corollary~\ref{cor:poweronhurewicz} this is congruent to
  \[
    \sum_{i=0}^{\infty} [c_i] \circ (b_1)^{i + 2n} u^{2(i+n)}.
  \]
  In particular, taking coefficients of both sides gives us that
  \[
    \mQ^{2k}([1] \hash ([x] \circ b_1^{\circ n})) \equiv [c_{k-n}]
    \circ b_1^{\circ (k+n)}
  \]
  in this quotient.
\end{proof}

\begin{rmk}
  The expression for $P(x)$ as a series in $\alpha$ is not unique due
  to the fact that it takes place in a quotient ring, and it is not
  immediately clear that the identity in this proposition is
  independent of this choice. However, any indeterminacy is a multiple
  of the identity $[2]_F(\alpha) = 0$, whose image in the Hopf ring
  under the total unstable invariant translates into an identity in
  terms of the Ravenel--Wilson relations.
\end{rmk}

We can now apply the results of Johnson--Noel from
Theorem~\ref{thm:johnsonnoel}, as well as
Corollary~\ref{cor:integralcomparison} and
Proposition~\ref{prop:DLstability}.
\begin{cor}
  \label{cor:gotcha}
  The Dyer--Lashof operations in $H_* SL_1(MU)$ satisfy
  \[
  \mQ^{10}([1]\hash ([x_2] \circ b_1^{\circ 2})) \equiv [1] \hash ([x_7]
  \circ b_1^{\circ 7})
  \]
  mod $\hash$-decomposables, $\circ$-decomposables, and the ideal
  $(b_2, b_3, \dots)$.

  The Dyer--Lashof operations in $\pi_* (H \sma_{MU} H)$ satisfy
  \[
  Q^{10} (\sigma x_2) = \sigma x_7.
  \]
\end{cor}

\begin{rmk}
  \label{rmk:noorientation}
  We can take a brief pause to sketch why no map $MU \to BP$ can be
  given the structure of a map of $E_7$-algebras at the prime $2$,
  extending \cite{noel-johnson-ptypical}. If it could, then we can
  obtain a map of $E_6$-algebras $H \sma_{MU} H \to H \sma_{BP} H$, on
  homotopy given by a map of exterior algebras
  $\Lambda[\sigma x_i] \to \Lambda[\sigma x_{2^i - 1}]$. However, this
  map would be zero on the element $\sigma x_2$ and nonzero on the
  element $\sigma x_7 = Q^{10} (\sigma x_2)$. (Here we use that
  $Q^{10}$, on a class in degree $5$, is realized by an operation for
  $E_6$-algebras---see Remark~\ref{rmk:topopsucks}.) This argument has
  been expanded in \cite{senger-bp}.
\end{rmk}

\section{Calculations with $MU$, $H\Bbb Z/2$, and $BP$}

In order to begin with more specific computations of secondary
operations, we will use the following convenient definitions.
\begin{defn}
  For a symbol $a$ and an integer $k$, we define $\mb P_H^{E_n}(a_k)$ to be the
  free $E_n$ $H$-algebra $\mb P_H^{E_n}(S^k)$, writing $a_k \in \pi_k
  \mb P_H^{E_n}(a_k)$ for the generator represented by the unit map
  $S^k \to \mb P_H^{E_n}(S^k)$.

  Similarly, we use the coproduct in $E_n$ $H$-algebras to define
  \[
  \mb P_H^{E_n}(a_{k_1}, b_{k_2},\dots) = \mb P_H^{E_n}(a_{k_1})
  \amalg \dots \mb P_H^{E_n}(b_{k_2}) \amalg \dots \cong \mb
  P_H^{E_n}(\vee S^{k_i})
  \]
  for a sequence $(a_{k_1}, b_{k_2},\dots)$.
  If a generator has a known, fixed, degree, we will leave off the
  subscript.
\end{defn}

\begin{defn}
  Let $\mathcal{D}$ be the category of $E_\infty$ $H$-algebras under
  $\mb P_H(x)$, where $x$ has degree $2$, and $\mathcal{D}^n$ the
  category of $E_n$ $H$-algebras under $\mb P_H^{E_n}(x)$.

  Let $\mathcal{C} = \ppa{\mathcal{D}}$ and $\mathcal{C}^n =
  \ppa{(\mathcal{D}^n)}$ as in Definition~\ref{def:ppa}.
\end{defn}

There are forgetful functors between these categories, using the
compatible maps $\mb P_H^{E_n} S^2 \to \mb P_H^{E_m} S^2$ that are
adjoint to the units $S^2 \to \mb P_H^{E_m} S^2$. The generator of $H_2
MU \cong \mb Z/2$ determines a map $\mb P_H(x) \to H \sma MU$ up to
equivalence, lifting it to an object of $\mathcal{C}$.

\subsection{Power operations for $MU$}

The $2$-primary power operations in $H_* MU$ are known by work of Kochman
\cite{kochman-dyerlashof}, but the following closed-form formula
is due to Priddy.
\begin{thm}[{\cite{priddy-dyerlashof}}]
  \label{thm:MU-DL}
  The Dyer--Lashof operations in $H_* MU \cong H_* BU$ are determined
  by the following identity:
  \[
  \sum Q^j b_k = \left(\sum_{n=k}^\infty \sum_{u=0}^k \binom{n-k+u-1}{u}
    b_{n+u} b_{k-u}\right)\left(\sum_{n=0}^\infty b_n\right)^{-1}
  \]
  Here $b_0 = 1$ by convention. In particular, we have
  \[
  \sum Q^j b_1 = \left(\sum_{n=1}^\infty (b_n b_1 + (n-1) b_{n+1})\right)
  \left(\sum_{n=0}^\infty b_n\right)^{-1}.
  \]
\end{thm}

This allows the following direct computation. (Compare
\cite[2.5]{priddy-dyerlashof}, which carries out this computation for
$MO$).
\begin{prop}
  \label{prop:MU-identities}
  We have the following Dyer--Lashof operations in $H_* MU$:
  \begin{align*} 
    Q^2 b_1 &= b_1^2\\
    Q^4 b_1 &= b_3 + b_1 b_2 + b_1^3\\
    Q^6 b_1 &= b_1^4\\
    Q^8 b_1 &= b_5 + b_1 b_4 + b_2 b_3 + b_1^2 b_3 + b_1 b_2^2 + b_1^3
    b_2 + b_1^5\\
    Q^{10} b_1 &= b_3^2 + b_1^2 b_2^2 + b_1^6\\
    Q^6 b_2 &= b_5 + b_1 b_4 + b_2 b_3 + b_1 b_2^2 \\
    Q^{10} b_2 &= b_1^2 b_5 + b_1^3 b_4 + b_1^2 b_2 b_3 + b_1^3 b_2^2
  \end{align*}
  In particular, the following identities hold:
  \begin{align*} 
    0 &= Q^6 b_1 + b_1^4\\
    0 &= Q^{10}b_1 + (Q^4 b_1)^2\\
    Q^6 b_2 &= Q^8 b_1 + b_1^2 Q^4 b_1\\
    0 &= Q^{10} b_2 + b_1^2 Q^6 b_2\\
  \end{align*}
\end{prop}

\subsection{Power operations for $H$}

The power operations in the dual Steenrod algebra are known by work of 
Steinberger.
\begin{thm}[{\cite[{III.2.2, III.2.4}]{bmms-hinfty}}]
  \label{thm:A-DL}
  The $2$-primary Dyer--Lashof operations in the dual Steenrod algebra
  satisfy the following identities:
  \[
  1 + \xi_1 + Q^1 \xi_1 + Q^2 \xi_1 + Q^3 \xi_1 + \dots = (1 + \xi_1 + \xi_2 + \dots)^{-1}
  \]
  \[
  Q^s \xx_i = \begin{cases}
    Q^{s+{2^i-2}} \xi_1&\text{if }s \equiv 0,-1 \mod 2^i,\\
    0&\text{otherwise.}
  \end{cases}
  \]
  \[
  Q^{2^i} \xx_i = \xx_{i+1}
  \]
\end{thm}

This, again, allowed direct computation.
\begin{prop}[{\cite[III.5]{bmms-hinfty}}]
  \label{prop:A-identities} 
  We have the following Dyer--Lashof operations in the $2$-primary
  dual Steenrod algebra:
  \begin{align*}
    Q^2 \xx_1 &= \xx_2\\
    Q^3 \xx_1 &= \xx_1^4\\
    Q^4 \xx_1 &= \xx_1^2 \xx_2\\
    Q^5 \xx_1 &= \xx_2^2\\
    Q^{16} \xx_4 &= \xx_5
  \end{align*}
  In particular, the Cartan formula implies that the following identities hold:
  \begin{align*} 
    0 &= Q^6 \xx_1^2 + \xx_1^8\\
    0 &= Q^8 \xx_1^2 + \xx_1^4 Q^4 \xx_1^2\\
    0 &= Q^{10} \xx_1^2 + (Q^4 \xx_1^2)^2
  \end{align*}
\end{prop}

\begin{rmk}
  \label{rmk:xixx}
  While the identity $Q^{16} \xx_4 = \xx_5$ is valid, the results of
  this paper only require us to know the easier statement that $Q^{16}
  \xi_4 \equiv \xi_5$ mod decomposable elements.
\end{rmk}

\subsection{Functional operations for $MU \to H\Bbb Z/2$}

Recall that the category $\mathcal{C}$ is the category of
$E_\infty$ $H$-algebras under $\mb P_H(x)$, where $x$ has degree $2$.
\begin{thm} 
  \label{thm:firstjuggle}
  Consider the maps
  \[
  \mb P_H(x,z_{14}) \too{\bcQ} \mb P_H(x,y_4) \too{f} H \sma MU
  \too{p} H \sma H
  \]
  in the category $\mathcal{C}$, where $\bcQ$ sends $z_{14}$ to
  $Q^{10} y_4 + x^2 Q^6 y_4$ and $f$ sends $(x,y_4)$ to
  $(b_1,b_2)$. Then a functional homotopy operation $\langle p, f,
  \bcQ\rangle$ is defined in $\mb P_H(x)$-algebras and satisfies
  \[
  \langle p, f, \bcQ \rangle \equiv \xi_4
  \]
  mod decomposables.
\end{thm}

\begin{proof}
  The identities $Q^{10} b_2 + b_1^2 Q^{6} b_2 = 0$ and $p(b_2) = 0$
  ensure that there is a homotopy commutative diagram of $E_\infty$ $\mb
  P_H(x)$-algebras over $H$:
  \[
  \xymatrix{
    \mb P_H(x,z_{14}) \ar[d]_-\bcQ \ar[r] & \mb P_H(x) \ar[d] \ar[dr]\\
    \mb P_H(x,y_4) \ar[r]^-f \ar[d]
    & H \sma MU \ar[r] \ar[d]_p & H \ar[d] \\
    \mb P_H(x) \ar[r] & H \sma H\ar[r]^-i & H \sma_{MU} H
  }
  \]
  In particular, $\bcQ$ is a map of augmented objects and $H \sma_{MU}
  H$ is a pointed object, ensuring that the secondary operation is
  defined. As a result, we can define $\langle p, f, \bcQ\rangle$ and
  apply the Peterson--Stein formula of
  Proposition~\ref{prop:bracketjuggles} to find that there is an
  identity
  \[
  \langle i, p, f\rangle \bcQ = i\langle p, f, \bcQ\rangle.
  \]
  (Note that there is no inversion in this Peterson--Stein formula
  because the target group is a vector space over $\mb F_2$.)

  The bracket $\langle i, p, f\rangle$ takes $y_4$, which maps under
  $f$ to the Hurewicz image $b_2$ of $x_2 \in \pi_2 MU$, to the
  suspension class $\sigma b_2$ up to indeterminacy by
  Proposition~\ref{prop:geomdetection}. The operation $\bcQ$ sends
  this to $Q^{10} (\sigma b_2) + x^2 Q^6(\sigma b_2) = Q^{10}(\sigma
  b_2)$ because $x$ acts by $0$ on $H \sma_{MU} H$. Then
  Corollary~\ref{cor:gotcha} implies that $Q^{10}(\sigma b_2) \equiv
  \sigma x_7$ mod decomposables, and the proof of
  Proposition~\ref{prop:mudetection} shoes that $\sigma x_7 \equiv
  i(\xi_4)$ mod decomposables. Thus we find that $i \langle p,f,\bcQ
  \rangle = \langle i, p,
  f\rangle \bcQ \equiv i(\xi_4)$ mod decomposables.

  The indeterminacy in the functional homotopy operation $\langle p, f,
  \bcQ\rangle$ consists of elements in the image of $p$ and elements
  in the image of $\sigma \bcQ$, which are of the form $Q^{10}(y_5') +
  \xi_1^2 Q^6(y_5')$. However, there are no indecomposables in the image
  of $p$ and no indecomposables in the dual Steenrod algebra in degree
  $5$, and so the indeterminacy consists completely of decomposable
  elements. The map $i$ is an isomorphism on homotopy in degree $15$
  mod decomposables, and hence  $\langle p, f, \bcQ\rangle \equiv
  \xi_4$ mod decomposables.
\end{proof}

\subsection{A secondary operation in the dual Steenrod algebra}

\begin{prop}
  \label{prop:bigrelation} 
  Suppose that $R$ is an $E_{12}$ $H$-algebra and $x \in
  \pi_2(R)$. Define the following classes:
  \begin{align*}
    y_5 &= Q^3 x\\
    y_7 &= Q^5 x\\
    y_9 &= Q^7 x\\
    y_{13} &= Q^{11} x\\
    y_8 &= Q^6 x + x^4\\
    y_{10} &= Q^8 x + x^2 Q^4 x\\
    y_{12} &= Q^{10}x + (Q^4 x)^2
  \end{align*}
  Then there is an identity
  \begin{align*}
    0 ={}& Q^{20} y_{10} + Q^{18} y_{12} + Q^{17} y_{13}  + x^4(Q^{12}
    y_{10}) + y_9^2 (Q^4 x)^2 +\\
    &y_7^2 Q^9Q^5 x + y_8^2 Q^8 Q^4 x + (Q^9 y_9) (Q^4x)^2 + (Q^{10} y_8) (Q^4 x)^2 + \\
    &y_5^2 (Q^{11} Q^7 x + Q^{10}Q^8 x + x^4 Q^6 Q^4 x)
  \end{align*}
\end{prop}

\begin{proof}
  The following table breaks this down term-by-term, substituting in
  the values of the $y_i$.
  \begin{align*} 
    Q^{20} y_{10} &= Q^{20} Q^8 x + Q^{20} (x^2 Q^4 x)\\%
    Q^{18} y_{12} &= Q^{18} Q^{10} x + Q^{18} ((Q^4 x)^2)\\%
    Q^{17} y_{13} &= Q^{17} Q^{11} x\\%
    x^4(Q^{12} y_{10}) &= x^4(Q^{12} Q^8 x) + x^8 Q^8 Q^4 x + x^4 (Q^3
    x)^2 Q^6 Q^4 x\\%
    y_9^2 (Q^4 x)^2 &= (Q^7 x)^2 (Q^4 x)^2\\%
    y_7^2 Q^9Q^5 x  &= (Q^5 x)^2 Q^9 Q^5 x\\%
    y_8^2 Q^8 Q^4 x &= (Q^6 x)^2 Q^8 Q^4 x + x^8 Q^8 Q^4 x\\%
    (Q^9 y_9) (Q^4x)^2 &= (Q^9 Q^7 x) (Q^4 x)^2 \\%
    (Q^{10} y_8) (Q^4 x)^2 &= (Q^{10} Q^6 x) (Q^4 x)^2 \\%
    y_5^2 (Q^{11} Q^7 x)&= (Q^3 x)^2(Q^{11} Q^7 x)\\%
    y_5^2(Q^{10} Q^8 x) &= (Q^3 x)^2 Q^{10} Q^8 x\\%
    y_5^2(x^4 Q^6 Q^4 x) &= (Q^3 x)^2 (x^4 Q^6 Q^4 x)%
  \end{align*}
  The reader who is interested in ensuring that these cancel is
  encouraged to do so with the aid of a pen. To assist this, we list
  the following needed identities deduced from the Cartan formula,
  Adem relations, and instability relations where appropriate.
  \begin{align*} %
    Q^{20} Q^8 x ={} &Q^{18} Q^{10} x + Q^{17} Q^{11} x\\%
    Q^{20}(x^2 Q^4 x) ={} &x^4 Q^{16} Q^4 x + (Q^3 x)^2 Q^{14} Q^4 x +
    (Q^4 x)^2 Q^{12} Q^4 x + \\
    &(Q^5 x)^2 Q^{10} Q^4 x + (Q^6 x)^2 Q^8 Q^4 x + (Q^7 x)^2 (Q^4 x)^2\\
    Q^{18}((Q^4 x)^2) ={} &0\\%
    x^4 Q^{16} Q^4 x ={} &x^4 Q^{12} Q^8 x\\%
    (Q^3 x)^2 Q^{14} Q^4 x ={} &(Q^3 x)^2 Q^{11} Q^7 x + (Q^3 x)^2 Q^{10}Q^8
    x\\%
    (Q^4 x)^2 Q^{12} Q^4 x ={} &(Q^4 x)^2 Q^{10} Q^6 x + (Q^4 x)^2
    Q^9 Q^7 x\\%
    (Q^5 x)^2 Q^{10} Q^4 x ={} &(Q^5 x)^2 Q^9 Q^5 x%
  \end{align*}
  To apply $Q^r$ to an element in degree $s$, as well as make use of
  the Adem relations, Cartan formula, and instability relations, we
  require the presence of an $E_n$-algebra for $n \geq r-s+2$. The
  greatest value of $n$ required from the equations above is when we
  take $Q^{20} y_{10}$, and in particular use additivity for $Q^{20}$,
  which requires an $E_{12}$-algebra.
\end{proof}

We can use this relation to build secondary operations.

\begin{prop}
  \label{prop:bigoperation}
  Suppose $n \geq 12$ and let $R$ be an object of $\mathcal{C}^n$,
  corresponding to an $E_n$ $H$-algebra with an element
  $x \in \pi_2(R)$, such that the classes $y_i$ of
  Proposition~\ref{prop:bigrelation} vanish. Then there is a secondary
  operation on $x$ given by
  $\langle x, Q \teth{h} R\rangle \in \pi_{31}
  R$. The indeterminacy in this secondary operation consists of
  elements of the form
  \[
    Q^{20} y'_{11}  + Q^{18} y'_{13} + Q^{17} y'_{14}
  \]
  and decomposables. This secondary operation is preserved by the
  forgetful functors $\mathcal{C}^m \to \mathcal{C}^n$ for $m > n$.
\end{prop}

\begin{proof}
  For any $n \geq 12$, Proposition~\ref{prop:bigrelation} describes a
  relation between homotopy operations, in the form of a homotopy
  commutative diagram of $E_{n}$ $H$-algebras
  \[
  \xymatrix{
    \mb P^{E_n}_H(z_{30}) \ar[d]_\epsilon \ar[r]^-{R} &
    \mb P^{E_n}_H(x, y_5, y_7, y_9, y_{13}, y_8, y_{10}, y_{12})
    \ar[d]^{Q} \\
    H \ar[r] & \mb P^{E_n}_H(x),
  }
  \]
  adjoint to a commutative diagram of $E_n$-algebras under $\mb
  P^{E_n}_H(x)$ of the form
  \[
  \xymatrix{
    \mb P^{E_n}_H(x,z_{30}) \ar[dr]_-\epsilon \ar[r]^-{R} &
    \mb P^{E_n}_H(x, y_5, y_7, y_9, y_{13}, y_8, y_{10}, y_{12}) \ar[d]^{Q} \\
    & \mb P^{E_n}_H(x).
  }
  \]
  Here the maps $Q$ and $R$ are defined by the equations of
  Proposition~\ref{prop:bigrelation}. The map $R$ is a map of
  augmented objects, the domain by the map $\epsilon$ sending $z_{30}$
  to $0$ and the range by the map sending all $y_i$ to zero. In
  particular, the homotopy commutativity of the above diagrams shows
  that there exists a tethering $Q \teth{h} R$ in the category
  $\mathcal{C}^n$.

  The indeterminacy in this secondary operation consists of elements
  in the image of the suspended operation
  \[
    \sigma R\co \mb P_H^{E_n}(x,z_{31}') \to \mb P^{E_n}_H(x, y_6',
    y_8', y_{10}', y_{14}', y_9', y_{11}', y_{12}').
  \]
  Proposition~\ref{prop:powerstability} implies that $\sigma R$ is
  given by
  \[
    z_{31}' \mapsto 
    Q^{20} y_{11}' + Q^{18} y_{13}' + Q^{17} y_{14}'  + x^4(Q^{12}
    y_{11}') + (Q^9 y_{10}') (Q^4x)^2 + (Q^{10} y_9') (Q^4 x)^2,
  \]
  since the other terms involve binary products that map to
  zero. However, the terms other than $Q^{20} y_{11}' + Q^{18} y_{13}'
  + Q^{17} y_{14}'$ always take decomposable values.
\end{proof}

\begin{prop}
  \label{prop:A-detects}
  In the $2$-primary dual Steenrod algebra
  \[
  H_* H \cong \mb F_2[\xi_1, \xi_2, \dots],
  \]
  viewed as the homotopy of the $H$-algebra $H \sma H$, the bracket
  $\langle \xi_1^2, Q, R\rangle$ is defined, and
  the indeterminacy is zero mod decomposables.
\end{prop}

\begin{proof}
  The Cartan formula for Dyer--Lashof operations immediately implies
  that $Q^{2k+1} (\xi_1^2) = 0$ for all $k$. The remaining identities
  \begin{align*}
    0 &= Q^6 \xi_1^2 + \xi_1^8\\
    0 &= Q^8 \xi_1^2 + \xi_1^4 Q^4 \xi_1^2\\
    0 &= Q^{10} \xi_1^2 + (Q^4 \xi_1^2)^2
  \end{align*}
  were determined in Proposition~\ref{prop:A-identities}. Therefore,
  $\xi_1^2 Q = 0$ and the bracket $\langle \xi_1^2, Q, R\rangle$ is
  defined.

  We now consider the indeterminacy. The indeterminacy is generated by
  adding the results of degree-29 homotopy operations applied to
  $\xi_1^2$, Dyer--Lashof operations applied to elements in degrees
  $11$, $13$, and $14$, and
  decomposables. Proposition~\ref{prop:gotalloperations} showed that
  all nonconstant homotopy operations are generated by multiplication,
  addition, and the operations $Q^n$, all of which preserve
  decomposables. The dual Steenrod algebra contains no indecomposables
  in degrees $11$, $13$, and $14$, and so any operation applied to
  such an element is decomposable.
\end{proof}

The operations $Q$ and $R$, while complex, can be related to simpler
operations using the following diagram.
\begin{prop}
  \label{prop:secondjuggle}
  Consider the maps
  \begin{align*}
    \mu\co &\mb P_H(x,y_5, y_7, y_9, y_{13}, y_8, y_{10}, y_{12}) \to \mb P_H(x,y_4) \\
    \nu\co &\mb P_H(x, z_{30}) \to \mb P_H(x,z_{14})\\
    \alpha\co &\mb P_H(x,z_{30}) \to \mb P_H(x,z_{15})\\
    \beta\co &\mb P_H(x,z_{15}) \to \mb P_H(x,y_4)
  \end{align*}
  of augmented objects, defined by the identities
  \begin{align*}
    \mu(y_i) &= 0\text{ for }i \neq 10\\
    \mu(y_{10}) &= Q^6 y_4\\
    \nu(z_{30}) &= Q^{16} z_{14}\\
    \alpha(z_{30}) &= (z_{15})^2\\
    \beta(z_{15}) &= Q^3 x Q^6 y_4
  \end{align*}
  Then there is an identity $\mu R = \bcQ \nu + \beta \alpha$ and a
  homotopy commutative diagram in $\mathcal{C}$ of the form
  \[
  \xymatrix{
    \mb P_H(x, y_i) \ar[d]^\mu \ar[r]^-{Q} &
    \mb P_H(x) \ar[d]_{b_1} \ar[dr]^{\xi_1^2} \\
    \mb P_H(x,y_4) \ar[r]_-f &
    H \sma MU \ar[r]^p & H \sma H,\\
  }
  \] 
  where $f$ and $\bcQ$ are from Theorem~\ref{thm:firstjuggle}.
\end{prop}

\begin{proof}
  It is classical that the map $p\co H_2 MU \to H_2 H$ takes $b_1$ to
  $\xi_1^2$, making the right-hand triangle commute.

  To verify that the map $\mu$ makes the square diagram commute in the
  homotopy category, we need to know that the Dyer--Lashof operations
  on $b_1$ satisfy
  \begin{align*} 
    0 &= Q^3 b_1\\
    0 &= Q^5 b_1\\
    0 &= Q^7 b_1\\
    0 &= Q^{11} b_1\\
    0 &= Q^6 b_1 + b_1^4\\
    0 &= Q^{10}b_1 + (Q^4 b_1)^2\\
    Q^6 b_2 &= Q^8 b_1 + b_1^2 Q^4 b_1
  \end{align*}
  The odd operations vanish automatically because $H_* MU$ is
  concentrated in even degrees, and the remaining three identities
  were proven in Proposition~\ref{prop:MU-identities}.

  Finally we need to verify the identity $\bcQ \nu = \beta \alpha +
  \mu R$. Using the definition of $\mu$ and the formula from
  Proposition~\ref{prop:bigrelation}
  for $R$, we find that
  \[
    \mu (R(z_{30})) = Q^{20} Q^6 y_4 + x^4 Q^{12} Q^6 y_4
  \]
  In the Adem relation
  $Q^{20} Q^6 = Q^{16} Q^{10} + Q^{14} Q^{12} + Q^{13} Q^{13}$, 
  the last two terms automatically vanish on classes in degree
  four. Therefore, we can continue to simplify, finding
  \begin{align*}
    \mu (R(z_{30})) &= Q^{16} Q^{10} y_4 + x^4 Q^{12} Q^6 y_4\\
    &= Q^{16} Q^{10}y_4 + Q^{16}(x^2 Q^6 y_4) + (Q^3 x)^2 (Q^6 y_4)^2\\
    &= Q^{16} (Q^{10}y_4 + x^2 Q^6 y_4) + (Q^3 x Q^6 y_4)^2\\
    &= Q^{16}(\bcQ (z_{14})) + \beta(z_{15})^2\\
    &= \bcQ \nu (z_{30}) + \beta \alpha(z_{30}),
  \end{align*}
  as desired.
\end{proof}

\begin{cor}
  \label{cor:A-detectsagain}
  In the dual Steenrod algebra, any element in the bracket $\langle
  \xi_1^2, Q, R\rangle$ is congruent to $\xi_5$
  mod decomposables.
\end{cor}

\begin{proof}
  We first observe that three types of elements in degree $31$ are
  decomposable in the dual Steenrod algebra.
  \begin{itemize}
  \item The first are elements in the image of $p\co H_* MU \to
    H_* H$: the only indecomposable element in the image of $p$ is $1
    \in H_0 H$.
  \item The second are elements in the image of $\sigma R$, which (as
    in Proposition~\ref{prop:bigoperation}) consists of multiples of
    Dyer--Lashof operations applied to elements in degrees $11$, $13$,
    and $14$. Degrees $11$, $13$ and $14$ contain no indecomposables,
    and so the Cartan formula for Dyer--Lashof operations implies that
    any elements in the image of $\sigma R$ are decomposable.
  \item The third are elements in the image of $\sigma(\bcQ \nu)$ or
    $\sigma(\beta \alpha)$, both of which are multiples of
    Dyer--Lashof operations applied to classes in degree $5$. Degree
    $5$ contains no indecomposables, and thus similarly the images of
    these elements are indecomposable.
  \end{itemize}

  Multiple applications of Proposition~\ref{prop:bracketjuggles} and
  Proposition~\ref{prop:addbrackets} give us the following string of
  identities.
  \begin{align*}
    \langle \xi_1^2, Q, R\rangle
    &= \langle p b_1, Q, R\rangle \\
    &\subset \langle p, b_1 Q, R\rangle \\
    &= \langle p, f \mu,  R\rangle \\
    &\supset \langle p, f, \mu R\rangle\\
    &= \langle p, f, \bcQ \nu + \beta \alpha\rangle\\
    &\subset \langle p, f, \bcQ \nu\rangle + \langle p, f,
      \beta \alpha\rangle \\
    &\supset \langle p, f, \bcQ\rangle \nu + \langle p, f, \beta\rangle\alpha.
  \end{align*}
  We note that in all of these brackets, the indeterminacy is
  contained in the three types mentioned above: the image of $p$, the
  image of $\sigma R$, and the images of $\sigma(\bcQ \nu)$ or
  $\sigma(\beta \alpha)$. It suffices to check at the local maxima
  for indeterminacy in this chain of containments: the brackets $\langle
  p, b_1 Q, R\rangle$ and $\langle p, f, \bcQ \nu\rangle + \langle p,
  f, \beta \alpha\rangle$. Therefore, if we work mod decomposables we
  get unambiguous values and these containments become equalities.
  We find
  \[
  \langle \xi_1^2, Q, R\rangle \equiv \langle p, f, \bcQ\rangle \nu +
  \langle p, f, \beta \rangle \alpha.
  \]

  By Theorem~\ref{thm:firstjuggle}, we have
  \[
  \langle p, f, \bcQ\rangle(\nu z_{30}) = Q^{16} (\langle p, f,
  \bcQ\rangle (z_{14})) \equiv Q^{16} \xi_4 \equiv \xi_5
  \]
  mod decomposables. On the other hand,
  \[
  \langle p, f, \beta\rangle \alpha(z_{30}) = (\langle p, f,
  \beta\rangle(z_{15}))^2
  \]
  which is automatically decomposable. Therefore, every element in
  $\langle \xi_1^2, Q, R\rangle$ is congruent to $\xi_5$ mod
  decomposables.
\end{proof}

\begin{thm}
  \label{thm:mainthm-redux}
  Suppose that $n \geq 12$ and $R$ is an $E_n$ ring spectrum with a
  map $g\co R \to H$ and an element $x \in H_2(R)$ such that $g(x) =
  \xi_1^2$ in $H_2 H$. If the element $x$ makes the classes $y_i$ of
  Proposition~\ref{prop:bigrelation} zero, then the map $H_{31} R
  \to H_{31} H$ has $\xi_5$ in its image mod decomposables.

  In particular, if $H_* R \to H_* H$ is injective through degree 13,
  this result holds.
\end{thm}

\begin{proof}
  Under these conditions, $H \sma R \to H \sma H$ is a map of
  $E_n$ $H$-algebras, and (up to equivalence) the map $\mb P_H^{E_n}(x) \to
  H \sma H$ lifts to a map $\mb P_H^{E_n}(x) \to H \sma R$. Thus, $H
  \sma R \to H \sma H$ can be lifted to a map in $\mathcal{C}^n$
  which, on homotopy groups, gives the map $g\co H_* R \to H_* H$.

  Then the secondary operation $\langle x, Q, R\rangle$ is defined and
  the map $g$ carries $\langle x, Q, R\rangle$ into a subset of
  $\langle \xi_1^2, Q, R \rangle$, all of whose elements are congruent
  to $\xi_5$ mod decomposables.
\end{proof}

\subsection{Application to the Brown--Peterson spectrum}

Using Theorem~\ref{thm:mainthm-redux}, we can now exclude the
existence of $E_n$-algebra structures on spectra related to the
Brown--Peterson spectrum. We first recall the homology of the
Brown--Peterson spectrum, dual to the cohomology described in
\cite{brown-peterson-bp}.
\begin{prop}
  The Brown--Peterson spectrum $BP$ is connective, with $\pi_0 BP
  \cong \mb Z_{(2)}$. The map $BP \to \HF_2$ induces an
  inclusion $H_* BP \into H_* H\mb F_2$ whose image is the subalgebra
  \[
    \mb F_2[\xi_1^2, \xi_2^2, \dots] \subset \mb F_2[\xi_1, \xi_2,
    \dots]
  \]
  of the dual Steenrod algebra. The image in positive degrees consists
  entirely of decomposables.
\end{prop}

Similarly, we have truncated Brown--Peterson spectra $\BP\langle
k\rangle$ and their generalized versions.
\begin{prop}[{\cite[4.3]{tmforientation}}]
  Any generalized truncated Brown--Peterson spectrum $BP\langle
  k\rangle$ is connective, with $\pi_0 BP\langle k\rangle= \mb 
  Z_{(2)}$. The map $BP\langle k\rangle \to \HF_2$
  induces an inclusion $H_* BP \into H_* \HF_2$ whose image is the
  subalgebra
  \[
    \mb F_2[\xi_1^2, \xi_2^2, \dots \xi^2_{k+1}, \xi_{k+2}, \xi_{k+3}, 
    \dots] \subset \mb F_2[\xi_1, \xi_2, \dots]
  \]
  of the dual Steenrod algebra. The image in positive degrees consists
  entirely of decomposables until dimension $2^{k+2}-2$.
\end{prop}
In particular, the element $\xi_5$ is not in the image mod
decomposables for $k \geq 4$. These spectra are also of finite type,
and their $2$-adic completions have the same homology groups.

By considering the cohomology in degree zero, we find that there is a
{\em unique} nontrivial map of spectra $BP \to H\mb F_2$, and
similarly for $BP\langle k\rangle$. (At odd primes, this map is unique
up to scalar.) As $E_n$-algebras have Postnikov towers, there is the
following consequence.
\begin{cor}
  If $BP$ or $BP\langle k\rangle$ admits the structure of an
  $E_n$-algebra, then the unique nontrivial map to $\HF_2$ lifts to
  a map of $E_n$-algebras.
\end{cor}

We can now apply Theorem~\ref{thm:mainthm-redux}.
\begin{thm}\label{thm:nobp-redux}
  The $2$-local Brown--Peterson spectrum $BP$, the (generalized)
  truncated Brown--Peterson spectra $BP\langle k\rangle$ for
  $k \geq 4$, and their $2$-adic completions do not admit the
  structure of $E_n$-algebras for any $12 \leq n \leq \infty$.
\end{thm}

\begin{rmk}
  The above results can also be applied to appropriate truncations in
  the Postnikov tower for $BP$.
\end{rmk}

\appendix
\section{Power operations in the Lazard ring}
\label{sec:powerops}

In this section we will extend Johnson--Noel's proof of
Theorem~\ref{thm:johnsonnoel} to a proof that works in torsion-free
quotients of the Lazard ring. The following calculations are
specialized to the prime $2$.

The power operation $P$ of Section~\ref{sec:cpxor} takes the form of a
natural transformation
\[
P\co MU^{2n}(X) \to MU^{4n}(X \times B\Sigma_2) \into
MU^{4n}(X)\pow{\alpha} / [2]_F(\alpha).
\]
Writing the $2$-series as
\[
[2]_F(\alpha) = \alpha \cdot \langle 2\rangle_F(\alpha),
\]
we have the following properties.
\begin{itemize}
\item The identity $P(uv) = P(u) P(v)$ holds.
\item The identity $P(u) \equiv u^2$ holds mod $\alpha$.
\item The identity $P(u+v) \equiv P(u) + P(v)$ holds mod $\langle
  2\rangle_F(\alpha)$. In particular, $P$ becomes a ring
  homomorphism in this quotient.
\item On the canonical orientation $x \in \wt{MU}^2(\mb{CP}^\infty)$,
  we have $P(x) = x(x +_F \alpha)$.
\end{itemize}

Let $L = MU^*$ be the Lazard ring, and define $g(x,\alpha) = x(x +_F
\alpha)$ in the power series ring $L\pow{x,\alpha}$.  Applying the
identities for $P$ to the spaces $(\mb{CP^\infty})^n$ and the natural
maps between them, we deduce the following.

\begin{prop}
  The map $P$ induces a ring homomorphism
  \[
  \Psi\co L \to L\pow{\alpha} / \langle 2\rangle_F(\alpha)
  \]
  and the power series $g(x,\alpha)$ defines an {\em isogeny} $F \to
  \Psi^*F$:
  \begin{equation}\label{eq:isogeny}
  g(x,\alpha) \mathop{+}_{\Psi^* F} g(y,\alpha) \equiv g(x +_F y,\alpha).
  \end{equation}
\end{prop}

The rings $L$ and $L\pow{\alpha} / \langle 2\rangle_F(\alpha)$ are
torsion-free, and so the formal group laws $F$ and $\Psi^* F$ have
logarithms:
\begin{align*}
  \ell_F(x) &= \sum \frac{\mb{CP}^{n-1} x^n}{n} &
  \ell_{\Psi^* F}(x, \alpha) &= \sum \frac{\Psi(\mb{CP}^{n-1}) x^n}{n}
\end{align*}
By choosing any lifts of $\Psi(\mb{CP}^n)$ to $L\pow{\alpha}$, we
can view these formulas as defining power series $(\ell_F)'(x) \in
L\pow{x}$ and $(\ell_{\Psi^* F})'(x, \alpha) \in L\pow{x,\alpha}$.

Taking derivatives of \eqref{eq:isogeny} with respect to $y$
and evaluating at $y=0$, we find
\[
\frac{g'(0,\alpha)}{(\ell_{\Psi^* F})'(g(x,\alpha),\alpha)} \equiv \frac{g'(x,\alpha)}{(\ell_F)'(x)}
\]
in $L\pow{x,\alpha} / \langle 2\rangle_F(\alpha)$, and thus
\begin{equation}
\label{eq:deriv}
g'(x,\alpha) (\ell_{\Psi^* F})'(g(x,\alpha),\alpha) = \alpha \cdot (\ell_F)'(x,\alpha) + h(x,\alpha)
\cdot \langle 2\rangle_F(\alpha)
\end{equation}
for some power series $h(x,\alpha) \in L\pow{x,\alpha}$.

We now substitute $x = \alpha y$ and observe that
\[
g(\alpha y,\alpha) = \alpha^2 k(y,\alpha),
\]
where the power series $k(y,\alpha)$ has the form $y + O(y^2)$. Hence,
there a composition inverse: a series $k^{-1}(y,\alpha)$ of the same
form such that $k(k^{-1}(y,\alpha),\alpha) = y$.

Substituting $x = \alpha y$ in to \eqref{eq:deriv}, we obtain an identity
\[
\alpha k'(y,\alpha) (\ell_{\Psi^* F})'(\alpha^2 k(y,\alpha),\alpha) = \alpha \cdot
(\ell_F)'(\alpha y) + h(\alpha y,\alpha)
\cdot \langle 2\rangle_F(\alpha)
\]
which can be simplified to the statement
\[
(\ell_{\Psi^* F})'(\alpha^2 y, \alpha) = 
(\ell_F)'(\alpha k^{-1}(y,\alpha))(k^{-1})'(y,\alpha) + \tilde h(y,\alpha)
\cdot \langle 2\rangle_F(\alpha)
\]
for some series $\tilde h(y,\alpha) \in L\pow{y,\alpha}$. If we write
$f_n(\alpha)$ for the coefficient of $y^n$ in $(\ell_F)'(\alpha
k^{-1}(y,\alpha))(k^{-1})'(y,\alpha)$, we then find that
\[
\Psi(\mb{CP}^n) \alpha^{2n} = f_n(\alpha) + \tilde h_n(\alpha) \cdot \langle 2\rangle_F(\alpha)
\]
for some series $\tilde h_n(\alpha)$. If $f\co L \to S$ is any ring
homomorphism, there is a degree-$2n$ polynomial $h_n(\alpha) \in
S[\alpha]$ such that
\[
f_n(\alpha) - h_n(\alpha) \cdot \langle 2\rangle_F(\alpha) \equiv 
f(\mb{CP}^n)^2
\]
in $S\pow{\alpha} / (\alpha^{2n+1})$. If $2$ is not a zero divisor in
the ring $S$, this determines $h_n(\alpha)$ uniquely and so it
can be calculated in $S$. We deduce that
\[
f(P(\mb{CP}^n)) \equiv \alpha^{-2n}\left(f_n(\alpha) - h_n(\alpha)\cdot \langle 2\rangle_F(\alpha)\right)
\]
in $S\pow{\alpha} / [2]_F(\alpha)$.

In particular, we may take $S = \mb Z[v_3]/(v_3^2)$, which has
logarithm $x + \tfrac{v_3}{2} x^8$. We can then expand out the
definitions in this ring.
\begin{align*}
  x +_F y &= x + y + \tfrac{v_3}{2} (x^8 + y^8 - (x+y)^8)\\
  \langle 2\rangle_F(\alpha) &= 2 - 127 v_3 \alpha^7\\
  (\ell_F)'(x) &= 1 + 4 v_3 x^7\\
  g(x,\alpha) &= \alpha x + x^2 + \tfrac{v_3}{2} (\alpha^8 + x^8 -
  (\alpha + x)^8)\\
  &= \alpha x + (1-4v_3\alpha^7)x^2 -14 v_3 \alpha^6 x^3 + O(x^4)\\
  k(y,\alpha) &= y + (1 - 4v_3\alpha^7)y^2 - 14 v_3 \alpha^7 y^3 + O(y^4)\\
  k^{-1}(y,\alpha) &= y + (4v_3\alpha^7 - 1)y^2 + (2 - 2v_3 \alpha^7) y^3
                     + O(y^4)\\
  (\ell_F)'(\alpha k^{-1}(y,\alpha)) &= 1 + O(y^7)\\
  (k^{-1})'(y,\alpha) &= 1 + (8v_3 \alpha^7 - 2)y + (6 - 6v_3
                        \alpha^7)y^2 + O(y^3)\\
  f_2(\alpha) &= 6 - 6 v_3 \alpha^7\\
  h_2(\alpha) &= 3\\
  f(P(\mb{CP}^2)) &= 375 v_3 \alpha^3 \\
  &\equiv v_3 \alpha^3.
\end{align*}
Here the last congruence follows because, in the ring
$S[\alpha]/[2]_F(\alpha)$,
\[
  2\alpha v_3 \equiv 127 v_3^2 \alpha^8 \equiv 0
\]
because $v_3^2 = 0$. Finally, $P(\mb{CP}^2)$, in
$BP_*\pow{\alpha} / [2]_F(\alpha)$ mod decomposables, can only involve
$v_3$, $v_4$, and higher, so we find
$P(\mb{CP}^2) \equiv v_3 \alpha^7$ mod decomposables and higher-order
terms in $\alpha$ as desired.

\bibliography{../masterbib}
\end{document}